\documentclass[11pt]{amsart}
\usepackage{fullpage}
\usepackage{amsmath, amssymb, amsthm}
\usepackage{comment}
\usepackage{graphicx}
\graphicspath{{./figs_ARM/}} 

\usepackage{color}
\usepackage[colorlinks=true,linkcolor=red,citecolor=blue,urlcolor=cyan]{hyperref}

\usepackage[ruled,vlined]{algorithm2e}
\usepackage{cleveref}

\newtheorem{thm}{Theorem}[section]
\newtheorem{cor}[thm]{Corollary}
\newtheorem{prop}[thm]{Proposition}
\newtheorem{lem}[thm]{Lemma}

\theoremstyle{remark}
\newtheorem{rem}[thm]{Remark}

\newcommand{\Em}{E}
\newcommand{\Var}{\mathrm{Var}}
\newcommand{\ubar}{\overline{u}_\rho}

\newcommand{\SAM}[1]{\textcolor{red}{\small {\sf SAM: #1}}}

\begin{document}
\title{Extremal Eigenvalues of Weighted Steklov Problems }\thanks{C.K. acknowledges support from NSF DMS-2208373 and DMS-2513176. }

\author{Chiu-Yen Kao}
\address{Department of Mathematical Sciences, Claremont McKenna College, Claremont, CA 91711}
\email{ckao@cmc.edu}

\author{Seyyed Abbas Mohammadi}
\address{Division of Mathematics, University of Dundee, Dundee DD1 4HN, United Kingdom; School of Computer Science and Applied Mathematics, University of The Witwatersrand, Braamfontein, 2000, Johannesburg, South Africa}
\email{smohammadi001@dundee.ac.uk;abbas.mohammadi@wits.ac.za}

\subjclass[2020]{35P15,  
49M05,  
49R05,  
65K10,  
65N25,  
35J25,  
}

\keywords{
Steklov eigenvalue problem, 
spectral optimization, 
non-convex optimization, 
variational methods, 
Fréchet differentiability, 
numerical algorithms}

\begin{abstract}
We study the optimization of Steklov eigenvalues with respect to a boundary density function $\rho$ on a bounded Lipschitz domain $\Omega \subset \mathbb{R}^N$. We investigate the minimization and maximization of $\lambda_k(\rho)$, the $k$th Steklov eigenvalue, over admissible densities satisfying pointwise bounds and a fixed integral constraint. Our analysis covers both first and higher-order eigenvalues and applies to domains 
$\Omega$ with general geometry and topology. We establish the existence of optimal solutions and provide structural characterizations: minimizers are bang--bang functions and may have disconnected support, while maximizers are not necessarily bang--bang.  
On circular domains, the minimization problem admits infinitely many minimizers generated by rotational symmetry, while the maximization problem has infinitely many distinct maximizers that are not symmetry-induced. We also show that the maps $\rho \mapsto \lambda_k(\rho)$ and $\rho \mapsto 1/\lambda_k(\rho)$ are generally neither convex nor concave, limiting the use of classical convex optimization tools.  
To address these challenges, we analyze the objective functional and introduce a Fréchet differentiable surrogate that enables the derivation of optimality conditions. We further design an efficient numerical algorithm, with experiments illustrating the difficulty of recovering optimal densities when they lack smoothness or exhibit oscillations.
\end{abstract}

\maketitle

\section{Introduction}

Optimizing eigenvalues of elliptic operators is a PDE-constrained design task with many real applications.
Examples include lowering the fundamental frequency of membranes and plates made of two materials under area or volume constraints \cite{antunes2013convex,chanillo2000symmetry,cox1991two,davoli2023spectral,krein1955,mohammadi2016extremal,mohammadi2019optimal};
reducing heat loss in systems with different thermal conductivities \cite{conca2012minimization,cox1996extremal,kang2020minimization,mohammadi2014optimal};
designing photonic crystals with chosen optical properties \cite{kao2005maximizing, kao2025semi};
optimizing the survival of a species in a habitat \cite{hintermuller2012principal,kao2008principal,mazari2022shape};
and decreasing the ground-state energy in nanostructures \cite{antunes2018nonlinear,mohammadi2012shape,mohammadi2016minimization1}, among others.
For overviews and further applications, see \cite{antunes2025optimization,bozorgnia2015optimal,bozorgnia2024infinity,bucur2005variational,henrot2006extremum,henrot2017shape,osting2014minimal,oudet2021computation,sayed2021maximization}.

The Steklov eigenvalue problem is an elliptic spectral problem, and its optimization has received considerable attention: it plays a role in spectral geometry \cite{colbois2024some,Girouard_2017}, is linked to free boundary minimal surfaces \cite{fraser2016sharp}, and—since the Steklov spectrum coincides with that of the Dirichlet-to-Neumann operator—arises in physical models and inverse problems such as electrical impedance tomography and cloaking \cite{kuznetsov2014legacy}.

In this paper, we study the optimization of Steklov eigenvalues with respect to boundary densities.
Let $\Omega\subset\mathbb{R}^N$, $N\ge2$, be bounded with Lipschitz boundary. For a boundary density $\rho\in L^\infty(\partial\Omega)$ with $\rho\ge0$ and $\rho\not\equiv0$, the weighted Steklov problem is
\begin{equation}\label{eq:mainpde}
\begin{cases}
\Delta u=0 & \text{in } \Omega,\\
\displaystyle \frac{\partial u}{\partial \mathbf n}=\lambda\,\rho\,u & \text{on } \partial\Omega,
\end{cases}
\end{equation}
where $\mathbf n$ is the outer unit normal. Equivalently,
\begin{equation}\label{eq:varform}
\int_\Omega \nabla u\cdot\nabla\phi\,d\mathbf x
= \lambda \int_{\partial\Omega}\rho\,u\,\phi\,dS
\qquad\text{for all }\phi\in H^1(\Omega),
\end{equation}
with $dS$ the surface measure  and boundary integrals are understood in the trace sense.


The Steklov spectrum is discrete and we list the eigenvalues (with multiplicity) as
\[
0=\lambda_0(\rho)<\lambda_1(\rho)\le\lambda_2(\rho)\le\cdots\to\infty.
\]
Let \(u_k:=u_{k,\rho}\) be an eigenfunction for \(\lambda_k(\rho)\).
Then the boundary traces \(\{u_{k,\rho}|_{\partial\Omega}\}_{k=0}^\infty\) form a complete
orthonormal system in the weighted space \(L^2(\partial\Omega;\rho\,dS)\), i.e.
\begin{equation}\label{e:rhoOrthonarmality}
  \int_{\partial\Omega}\rho\,u_{i,\rho}\,u_{j,\rho}\,dS=\delta_{ij}.
\end{equation}
In particular, the first eigenfunction \(u_{0,\rho}\) corresponding to \(\lambda_0(\rho)\) is constant.

The eigenvalues admit the variational characterization
\begin{equation}\label{eq:lambdavarform1}
  \lambda_{k}(\rho)
  = \min_{\substack{E_{k}\subset H^1(\Omega)\\ \dim E_{k}=k,\ E_k\subset H^1_{\rho,0}}}
    \ \max_{0\neq v \in E_{k}}
    \frac{\int_{\Omega}\left|\nabla v\right|^{2}\,d\mathbf{x}}{\int_{\partial\Omega}\rho\, v^{2}\,dS},
  \qquad k=1,2,\dots
\end{equation}
where
\[
H^1_{\rho,0}:=\Bigl\{\,v\in H^1(\Omega):\ \int_{\partial\Omega} \rho\, v\, dS = 0\,\Bigr\}.
\]
Equivalently, defining
\begin{equation}\label{eq:Akrho}
  \mathcal{A}^k_\rho := \Bigl\{\,v\in H^1(\Omega):\ v\neq 0,\ \int_{\partial\Omega}\rho\, v\, u_{j,\rho}\,dS=0\ \text{for } 0\le j\le k-1\,\Bigr\},
\end{equation}
we also have
\begin{equation}\label{eq:lambdavarform2}
  \lambda_{k}(\rho)
  = \min_{v\in\mathcal{A}^k_\rho}
    \frac{\int_{\Omega}\left|\nabla v\right|^{2}\,d\mathbf{x}}{\int_{\partial\Omega}\rho\, v^{2}\,dS},
  \qquad k=1,2,\dots,
\end{equation}
see, e.g., \cite[p.~97]{bandle1980isoperimetric} and \cite[Section~7.3]{henrot2006extremum}.

When $\rho\equiv 1$, the Steklov spectrum coincides with the spectrum of the Dirichlet-to-Neumann operator $\Gamma \colon H^{\frac1 2}(\partial \Omega) \to H^{- \frac 1 2}(\partial \Omega)$,  given by the formula $\Gamma w =  \partial_ {\mathbf{n}} ( \mathcal H w)$, where $\mathcal H w$ denotes the unique harmonic extension of $ w \in H^{\frac1 2}(\partial \Omega)$ to $\Omega$. The restrictions of the Steklov eigenfunctions to the boundary, 
$\{ u_j |_{\partial \Omega} \}_{j=0}^\infty \subset {C^\infty(\partial \Omega)}$, 
form   a complete orthonormal basis of $L^2(\partial \Omega)$. Discussions on recent developments on the Steklov eigenvalue problems can be found in  \cite{colbois2024some,Girouard_2017}. 

Lately, study of maximal Steklov eigenvalues and their properties attracts a lot of attention, see for example \cite{colbois2024some}. One of the problems that was studied intensively is to find upper bounds for   Steklov eigenvalues. 
It was proved by Weinstock in 1954~\cite{Weinstock_1954} that, among simply connected planar domains and for densities \( \rho \) with fixed mass \( \int_{\partial \Omega} \rho \, dS = \gamma \), the first Steklov eigenvalue satisfies \( \lambda_{1}(\rho) \leq 2\pi / \gamma \), with equality attained when \( \Omega \) is a disk, and \( \rho \) is constant.
 Furthermore, the first Steklov eigenvalue of the disk with the constant density $\rho \equiv 1$ has multiplicity two and $\lambda_1=\lambda_2 = 1/R$ where $R$ is the radius of the disk.  Later Girouard and Polterovich \cite{girouard2010} proved that, among simply connected planar domains, $\lambda_{k}({\rho}) \le 2\pi k/\gamma$  and the bound is sharp and attained by a sequence of simply connected domains degenerating into a disjoint union of $k$ identical balls. Among the rotationally symmetric annuli,  $\lambda_{1}({\rho}) \gamma\le \tilde{\Lambda} \approx10.47478$ with the equality attained on an annulus with an inner radius $r_{s}\approx 0.090776$, an outer radius $r = 1$, and constant ratio of density $\rho_s/\rho_1 \approx 11.01609$ where $\rho_s$ and $\rho_1$ are density functions on the inner and outer radii, respectively \cite{oudet2021computation}. 

Fix $k\ge 1$. We study
\begin{equation}\label{eq:minp}
  \min_{\rho\in\mathcal M}\ \lambda_k(\rho),
\end{equation}
\begin{equation}\label{eq:maxp}
  \max_{\rho\in\mathcal M}\ \lambda_k(\rho),
\end{equation}
where
\[
\mathcal M=\Bigl\{\rho\in L^\infty(\partial\Omega):\ \alpha\le \rho\le \beta\ \text{a.e. on }\partial\Omega,\ 
\int_{\partial\Omega}\rho\,dS=\gamma\Bigr\},
\]
with fixed parameters $0<\alpha<\beta$ and $\alpha\,|\partial\Omega|<\gamma<\beta\,|\partial\Omega|$. 
Here $|\cdot|$ denotes Lebesgue measure on $\Omega$ and surface (Hausdorff) measure on $\partial\Omega$.

It is well known that $\mathcal{M}$ is the weak-$^\ast$ closure of the set of step (bang--bang) functions
\[
\mathcal{N}
=\Bigl\{\rho\in L^\infty(\partial\Omega):\ \rho=\alpha+(\beta-\alpha)\chi_D,\ D\subset\partial\Omega,\ |D|=A\Bigr\},
\]
where
\begin{equation}\label{eq:Aformula}
    A=\frac{\gamma-\alpha\,|\partial\Omega|}{\beta-\alpha}.
\end{equation}
Moreover, $\mathcal{M}$ is weak-$^\ast$ compact and convex, and its extreme points are precisely the functions in $\mathcal{N}$; see \cite{friedland1977extremal,henrotpierre2018shape,mazari2024qualitative}.

The Steklov eigenvalue problem appears in models where a quantity (such as mass, heat, or charge) is concentrated on the boundary rather than inside the domain.
A physical interpretation of optimization problems \eqref{eq:minp}-\eqref{eq:maxp} arises in the analysis of elastic membranes, see \cite[Page 95]{bandle1980isoperimetric} and \cite{lamberti2015viewing}. Assume we want to build a membrane of prescribed shape $\Omega$ with  total mass $\gamma$  where the mass is concentrated on the boundary. Moreover, the membrane on the boundary is made out of two different materials with densities $\alpha$ and $\beta$.  Solutions of \eqref{eq:minp}-\eqref{eq:maxp} provide a pattern for distributing these materials on the boundary in a such a way that  frequencies of the resulting membrane are optimal. Optimizing frequencies of membranes with different boundary conditions have been studied by several authors, see \cite{chanillo2000symmetry, henrot2006extremum, kao2013efficient, mohammadi2019optimal}, to name just a few.


While most prior works focus on maximizing Steklov eigenvalues, particularly the first eigenvalue, significantly less attention has been paid to minimization problems or higher-order eigenvalues. 
In contrast, this paper makes the following contributions: 1) We study both the {minimization and maximization} of the quantity \( \lambda_k(\rho)\gamma \) over a class of bounded domains with Lipschitz boundaries, with no assumptions of convexity or simple connectivity on the domain $\Omega$ and no convexity or concavity assumption on the functional with respect to $\rho$.  2) We provide numerical and analytical evidence that the Steklov eigenvalues \( \lambda_k(\rho) \) and their reciprocals \( 1/\lambda_k(\rho),\) $k\geq 1$, are generally neither concave nor convex over the admissible set \( \mathcal{M} \). This highlights limitation of classical convex optimization tools and justifies the need for tailored optimization algorithms. This part of our study was motivated by related results in \cite[ Theorem 9.1.3]{henrot2006extremum}.
3) We establish the existence of optimal solutions and study their structure. In particular, we show that there are  bang–bang minimizers and we find some of  the topological properties of the high–density set for the minimizers. For the maximization problem, the optimal density is not necessarily bang–bang; on the circle we in fact find many maximizers with distinct profiles. 4) Although the eigenvalues are not differentiable in general, we introduce a differentiable surrogate functional and use it to derive the optimality conditions. We implement a numerical scheme consistent with the derived optimality conditions.  5) Our approach integrates both analytical insight and computational techniques, making it applicable to a broad range of geometries where traditional tools fail. 6) Finally, to the best of our knowledge, the \emph{minimization} of \( \lambda_k(\rho)\gamma \) over boundary densities with fixed mass and pointwise bounds \( \alpha \le \rho \le \beta \) on general Lipschitz domains has not been investigated before; previous studies have focused mainly on maximization, shape optimization, or related Robin and nonlinear Steklov problems, see for example \cite{colbois2024some,emamizadeh2011rearrangements,mazari2022qualitative}.

In Section~\ref{sec:pre}, we investigate analytical properties of the Steklov eigenvalue \(\lambda_k(\rho)\), including continuity and differentiability with respect to the density function \(\rho\). Section~\ref{sec:convexity} addresses the (non-)convexity of the objective functional, a property that significantly impacts the choice of optimization methods. 
Section~\ref{sec:minprob} is devoted to the minimization problem: we prove the existence of a bang-bang optimizer and analyze its topological and geometric structure. In Section~\ref{sec:maxprob}, we turn to the maximization problem, establish existence of an optimal solution, and identify it explicitly when the domain is a disk and $k=1$. Section~\ref{sec:numeric} presents numerical algorithms designed to solve the optimization problems and the numerical results, and we conclude in Section~\ref{sec:conc} with a summary of findings and future directions.

\section{Properties of the functional $\lambda_k(\rho)$ and eigenfunction $u_{k,\rho}$}\label{sec:pre}

In this section, we study the eigenvalues as a functional over the set $\mathcal{M}$ and establish regularity results for the corresponding eigenfunction $u_k$.  

We begin by presenting regularity results for the eigenfunctions. The following lemma is needed, and its proof can be found in \cite{griepentrog2001linear} and  \cite[Theorem 5.8]{manzoni2021optimal}.
\begin{lem}\label{lem:reglem}
Let $\Omega\subset\mathbb{R}^N$ be a bounded Lipschitz domain with $N\ge2$. 
Suppose $u\in H^1(\Omega)$ is the weak solution of
\begin{equation}\label{eq:fgpde}
    \begin{cases}
		\Delta u + u = f & \text{in } \Omega, \\
		\frac{\partial u}{\partial \mathbf{n}} = g & \text{on } \partial \Omega,
	\end{cases}
\end{equation}
that is,
\[
\int_{\Omega} \nabla u\cdot \nabla \phi \, d\mathbf{x} 
+ \int_{\Omega} u \phi \, d\mathbf{x} 
= \int_{\Omega} f \phi \, d\mathbf{x} 
+ \int_{\partial\Omega} g \phi \, dS, 
\qquad \forall \phi\in H^{1}(\Omega).
\]
If $f\in L^p(\Omega)$ with $p>\tfrac N2$ and $g\in L^q(\partial\Omega)$ with $q> N-1$, then $u\in C(\overline{\Omega})$ and
\[
\|u\|_{C(\overline{\Omega})} \leq C \left( \|f\|_{L^p(\Omega)} + \|g\|_{L^q(\partial\Omega)} \right),
\]
where $C$ depends on $\Omega$, $p$, and $q$.
\end{lem}

Now, we are ready to state the following lemma on the regularity of the eigenfunction \( u_k \).
\begin{lem}\label{lem:requk}
Let $\Omega\subset\mathbb{R}^N$ be a bounded Lipschitz domain and fix $k\ge 1$. The eigenfunction $u_k$ of \eqref{eq:mainpde} belongs to $C^\infty(\Omega)$. Moreover, if $N\le 3$, then
$u_k \in C(\overline{\Omega})$, and
\[
\|u_k\|_{C(\overline{\Omega})} \le C \left( \|u_k\|_{L^2(\Omega)} + \|u_k\|_{L^2(\partial\Omega)} \right),
\]
where constant $C$ depends on $\Omega$ and $p$.
\end{lem}

\begin{proof}
The first assertion is standard; see, for instance, \cite[Corollary 8.11]{gilbarg2015elliptic}.

For $N=2$, since $u_k\in H^1(\Omega)$, we have $u_k\in L^2(\Omega)$ and its trace belongs to $L^2(\partial\Omega)$. Setting $f=u_k$ and $g=\lambda_k\rho\,u_k$ in Lemma~\ref{lem:reglem}, we observe that $u_k$ satisfies \eqref{eq:fgpde}; hence $u_k\in C(\overline{\Omega})$, and the stated estimate follows.

For $N=3$, by \cite[Lemma 3.11]{nittka2011regularity}, 
we have $u_k\in L^p(\Omega)$ and  $u_k\in L^p(\partial\Omega)$ for all $p>1$. In particular, \cite[Lemma 3.11]{nittka2011regularity} yields that
\begin{equation}\label{ineq:lp2}
    \|u_k\|_{L^p(\Omega)}+  \|u_k\|_{L^p(\partial\Omega)} \leq C \left( \|u_k\|_{L^2(\Omega)}+  \|u_k\|_{L^2(\partial\Omega)}\right),
\end{equation}
where $C$ does not depend on $u_k$.

Taking $p>2$ and again setting $f=u_k$ and $g=\lambda_k\rho\,u_k$ in Lemma~\ref{lem:reglem}, we conclude that $u_k\in C(\overline{\Omega})$ and obtain the inequality in Lemma ~\ref{lem:reglem} for $ p=q>2$. Now, in view of  \eqref{ineq:lp2}, we obtain the assertion of this lemma.

\end{proof}


Next, we show the {continuity and differentiability of $\lambda_k(\rho)$.}
Recall the orthonormality \eqref{e:rhoOrthonarmality}. 
\begin{lem}\label{lem:weakcon}
Fix $k\ge 1$. Set
\[
U:=\bigl\{\rho\in L^\infty(\partial\Omega):\; \alpha - \epsilon \leq\rho(\mathbf{x}) \le \beta + \epsilon, \:\:\: \text{a.e. on }\partial\Omega\bigr\},
\]
where $\epsilon>0$ with $\alpha - \epsilon>0$.
Consider $\lambda_k:U\to\mathbb R$. Assume   $\lambda_{k}(\rho) $ is of multiplicity $l+1,$ $l\geq 0$, i.e. $ \lambda:=\lambda_{k}(\rho) = \cdots =\lambda_{k+l}(\rho)$.
\begin{enumerate}
\item[(i)] (\emph{Weak-$^\ast$ continuity})
If $\{\rho_i\}_1^\infty \subset U$ and  $\rho_i \stackrel{\ast}{\rightharpoonup} \rho$ in $L^\infty(\partial\Omega)$, then
$\lambda_k(\rho_i)\to \lambda_k(\rho)=\lambda$.
Moreover, after passing to a subsequence, there exist eigenfunctions
$\{u_{k+j,\rho_i}\}_{j=0}^l\subset H^1(\Omega)$ associated with
$\{\lambda_{k+j}(\rho_i)\}_{j=0}^l$ such that
\[
\int_{\partial\Omega}\rho_i\,u_{k+j,\rho_i}\,u_{k+s,\rho_i}\,dS=\delta_{js},
\qquad 0\le j,s\le l,
\]
and, for each $j=0,\dots,l$,
\[
u_{k+j,\rho_i}\rightharpoonup v_j \ \text{ in } H^1(\Omega),
\qquad
u_{k+j,\rho_i}\to v_j \ \text{ in } L^2(\partial\Omega),
\]
where $\{v_0,\dots,v_l\}$ is a $\rho$-orthonormal family of eigenfunctions for $\lambda$
and hence a $\rho$-orthonormal basis of the eigenspace $E_\lambda$.
\item[(ii)] (\emph{G\^ateaux derivative of a cluster sum})
  Define
\[
\Lambda_{k,l}(\rho):=\sum_{j=0}^l \lambda_{k+j}(\rho).
\]
For any $h\in L^\infty(\partial\Omega)$ with $\rho_t:=\rho+t h\in U$ for $|t|$ small,
$\Lambda_{k,l}$ is G\^ateaux differentiable at $\rho$ and
\[
\bigl(\Lambda'_{k,l}(\rho),h\bigr)
= -\,\lambda_k(\rho)\int_{\partial\Omega} h\,\Bigl(\sum_{j=0}^l u_{k+j,\rho}^2\Bigr)\,dS.
\]

\item[(iii)] (\emph{Fr\'echet differentiability})
If $N\le 3$, then $\Lambda_{k,l}(\rho)$ is Fr\'echet differentiable at $\rho$ relative to $U$ in $L^2(\partial \Omega)$-topology, i.e.
\[
\Lambda_{k,l}(\rho+h)-\Lambda_{k,l}(\rho)
=\bigl(\Lambda'_{k,l}(\rho),h\bigr)+o\bigl(\|h\|_{L^2(\partial\Omega)}\bigr).
\]
\end{enumerate}
\end{lem}
\begin{proof}
    \emph{(i)} Let us define  the resolvent operator $T: L^2(\partial \Omega)\to  L^2(\partial \Omega) $ such that $T(f)=u_f$ where $u_f$ is a solution of the following differential equation
	\begin{equation}\label{eq:resolventpde}
		\left\{
		\begin{array}{ccc}
			\Delta u= 0   &\text{in}&\Omega, \\
			\frac{\partial u}{\partial \mathbf{n}}+\rho u=\rho f   &\text{on}&\partial \Omega,
		\end{array}
		\right.
	\end{equation}
	and the values of $u_f$ on $\partial \Omega$ should be understood in the trace sense. The variational formula for a solution of this equation is 
	\begin{equation}\label{eq:resolvarform}
		\int_{\Omega} \nabla u\cdot \nabla \phi d\mathbf{x}+\int_{\partial\Omega}\rho u \phi dS= \int_{\partial\Omega}\rho f \phi dS,\quad\text{for all}\quad \phi\in H^{1}(\Omega),
	\end{equation}
	and then it is straightforward to show that this equation has a unique solution due to the Lax-Milgram Lemma  \cite[Theorem 8.3.4]{atkinson2009theoretical}. This reveals that the operator $T$ is well-defined. Moreover, it is easy to check that $T$ is a bounded linear operator. In addition, it is positive due to the ellipticity of \eqref{eq:resolventpde} and also symmetric because of the variational form of \eqref{eq:resolvarform}.  The operator $T$ is compact since $H^1(\Omega)$ compactly embedded in $L^2(\partial \Omega)$.
	
	Fix $i\in \mathbb{N}$, similarly, we can define the resolvant operator $T_i: L^2(\partial \Omega)\to  L^2(\partial \Omega) $ such that $T_i(f)=u_i$ where $u_i$ is a solution of 
	\begin{equation}\label{eq:resolventpdei}
		\left\{
		\begin{array}{ccc}
			\Delta u= 0   &\text{in}&\Omega, \\
			\frac{\partial u}{\partial \mathbf{n}}+\rho_i u=\rho_i f   &\text{on}&\partial \Omega.
		\end{array}
		\right.
	\end{equation}
	The variational form of this equation is 
	\begin{equation}\label{eq:resolvarformi}
		\int_{\Omega} \nabla u\cdot \nabla \phi d\mathbf{x}+\int_{\partial\Omega}\rho_i u \phi dS= \int_{\partial\Omega}\rho_i f \phi dS,\quad\text{for all}\quad \phi\in H^{1}(\Omega).
	\end{equation}
	Similar to that of operator $T$, one can infer that $T_i$ is a symmetric, positive and compact operator. 
	Then, in view of \cite[Theorem 2.3.1]{henrot2006extremum}  we have that
	\begin{equation}\label{ineq:sup}
		\left\lvert \mu_k(T)-\mu_k({T_i}) \right\rvert \leq \sup_{f\in L^2(\partial \Omega)} \frac{\|u_f-u_i\|_{L^2(\partial \Omega)}}{\|f\|_{L^2(\partial \Omega)}},
	\end{equation}
	where $\mu_k(T)$ and  $\mu_k(T_i)$ is the $k$-th respective
	eigenvalues.
	
	Applying \eqref{eq:resolvarformi}, \cite[Corollary 6.2]{auchmuty2005steklov}, the  Cauchy-Schwarz inequality and the trace theorem, we have
	\begin{align*}
		c\|u_i\|^2_{H^1(\Omega)}&\leq \int_{\Omega} \lvert\nabla u_i\rvert^2 d\mathbf{x}+ (\alpha -\epsilon)\int_{\partial\Omega} u_i^2  dS \leq \int_{\Omega} \lvert\nabla u_i\rvert^2 d\mathbf{x}+ \int_{\partial\Omega} \rho_i u_i^2  dS\\
		&\leq   (\beta +\epsilon) \|f\|_{L^2(\partial\Omega)}   \|u_i\|_{L^2(\partial\Omega)}\leq (\beta +\epsilon)   \|f\|_{L^2(\partial\Omega)}   \|u_i\|_{H^1(\Omega)},
	\end{align*}
	where $c>0$ and independent of $u_i$. This reveals that, the sequence $\|u_i\|^2_{H^1(\Omega)}$ is bounded from above. Thus, there is a subsequence for $\{u_i\}_1^\infty$, still  denoted by $\{u_i\}_1^\infty$ for simplicity,  and $u\in H^1(\Omega)$ such that
	\begin{equation}\label{eq:uiconv}
		 u_{i} \rightharpoonup  u, \,\,\,\text{weakly in }\,\,\, H^1(\Omega),\,\,\, \,\, \text{and}\,\, u_{i}\to u,\,\,\,\text{in}\,\, L^2(\partial\Omega).
	\end{equation}
	Note  $\{\rho_i\}_1^\infty\subset U$ and so there is subsequence, still denoted by $\{\rho_i\}_1^\infty$, such that $\rho_i \stackrel{\ast}{\rightharpoonup} \rho$ in   $L^\infty(\partial\Omega)$.
	Then, in view of \eqref{eq:resolvarformi}, $\rho_i \stackrel{\ast}{\rightharpoonup} \rho$ and \eqref{eq:uiconv}, passing $i\to \infty$ we obtain
	\[\int_{\Omega} \nabla u\cdot \nabla \phi d\mathbf{x}+\int_{\partial\Omega}\rho u \phi dS= \int_{\partial\Omega}\rho f \phi dS,\quad\text{for all}\quad \phi\in H^{1}(\Omega).\]
	Indeed, we have $u=u_f$ and consequently we can say $u_{i}\to u_f$ in  $L^2(\partial\Omega).$ Therefore, one can conclude $\lvert \mu_k(T)-\mu_k({T_i})\rvert\to 0 $ due to \eqref{ineq:sup}.
	
	For the $k$-th eigenvalues, we have 
	\[u_f=T(f)=\mu_k(T)f,\quad u_i=T_i(f)=\mu_k(T_i)f,\]
	and then \eqref{eq:resolvarform} and  \eqref{eq:resolvarformi} yield that
	\[\int_{\Omega} \nabla u_f\cdot \nabla \phi d\mathbf{x}+\int_{\partial\Omega}\rho u_f \phi dS= \int_{\partial\Omega}\rho \left(\frac{1}{\mu_k(T)}\right) u_f \phi dS,\quad\text{for all}\quad \phi\in H^{1}(\Omega),\]
	\[\int_{\Omega} \nabla u_i\cdot \nabla \phi d\mathbf{x}+\int_{\partial\Omega}\rho_i u_i \phi dS= \int_{\partial\Omega}\rho_i \left(\frac{1}{\mu_k(T_i)}\right) u_i \phi dS,\quad\text{for all}\quad \phi\in H^{1}(\Omega),\]
	which leads to the fact that $\lambda_{k}(\rho)= (1/\mu_k(T)-1)$ and $\lambda_{k}(\rho_i)= (1/\mu_k(T_i)-1)$. This means that $\lambda_{k}(\rho_i)\to \lambda_{k}(\rho)$ for any $k\geq 1$ as $\rho_i \stackrel{\ast}{\rightharpoonup} \rho$.
    
    	
For each $i$, choose eigenfunctions $\{u_{k+j,\rho_i}\}_{j=0}^l \subset H^1(\Omega)$ such that
\begin{equation}\label{eq:rhoi-orthonormal}
  \int_{\partial\Omega}\rho_i\,u_{k+j,\rho_i}\,u_{k+s,\rho_i}\,dS=\delta_{js},
  \qquad 0\le j,s\le l.
\end{equation}
By a reasoning similar to the above, one can find a uniform $H^1$ bound for these eigenfunctions, and  after extracting a subsequence, for each $j$ we have
\begin{equation}\label{eq:conv-us}
  u_{k+j,\rho_i}\rightharpoonup v_j \ \text{ in }H^1(\Omega),
  \qquad
  u_{k+j,\rho_i}\to v_j \ \text{ in }L^2(\partial\Omega).
\end{equation}
For any $\phi\in H^1(\Omega)$ the variational formulation \eqref{eq:varform} gives
\[
\int_\Omega \nabla u_{k+j,\rho_i}\!\cdot\!\nabla\phi\,dx
=\lambda_{k+j}(\rho_i)\int_{\partial\Omega}\rho_i\,u_{k+j,\rho_i}\,\phi\,dS.
\]
Since $\lambda_{k+j}(\rho_i)\to\lambda$ as $\rho_i \stackrel{\ast}{\rightharpoonup} \rho$ in $L^\infty(\partial\Omega)$, and \eqref{eq:conv-us} holds, passing to the limit yields
\begin{equation}\label{eq:limit-eig}
\int_\Omega \nabla v_j \!\cdot\! \nabla \phi\, dx
=\lambda \int_{\partial\Omega} \rho\, v_j\, \phi\, dS,
\qquad \forall\, \phi \in H^1(\Omega).
\end{equation}
so each $v_j$ is an eigenfunction for $\lambda$ at $\rho$.

Next, pass to the limit in \eqref{eq:rhoi-orthonormal} and
using $\rho_i \stackrel{\ast}{\rightharpoonup} \rho$,
\[
\int_{\partial\Omega}\rho\,v_j\,v_s\,dS
=\lim_{i\to\infty}\int_{\partial\Omega}\rho_i\,u_{k+j,\rho_i}\,u_{k+s,\rho_i}\,dS
=\delta_{js},\qquad 0\le j,s\le l.
\]
Thus $\{v_0,\dots,v_l\}$ is a $\rho$-orthonormal family of eigenfunctions for $\lambda$. Since the eigenspace $E_\lambda$ has dimension $l+1$, this family is a $\rho$-orthonormal basis. In particular, when $\lambda_k(\rho)$ is simple ($l=0$), we may take $v_0=u_{k,\rho}$ (up to sign).
\\
\emph{(ii)} Let $h\in L^\infty(\partial\Omega)$ and set $\rho_t=\rho+t h\in U$ for $|t|$ small.
Fix a $\rho$-orthonormal basis of the eigenspace $E_\lambda$:
\[
u_r:=u_{k+r,\rho},\qquad r=0,\dots,l,
\qquad
\int_{\partial\Omega}\rho\,u_r u_s\,dS=\delta_{rs}.
\]
For each $t$, choose eigenfunctions $\{u_{k+j,t}\}_{j=0}^l$ corresponding to
$\{\lambda_k(\rho_t),\dots,\lambda_{k+l}(\rho_t)\}$ and normalized by
\[
\int_{\partial\Omega}\rho_t\,u_{k+j,t}\,u_{k+m,t}\,dS=\delta_{jm},
\qquad 0\le j,m\le l.
\]

Fix $r\in\{0,\dots,l\}$ and $j\in\{0,\dots,l\}$. Test  weak formulation \eqref{eq:varform} at $\rho_t$
for $u_{k+j,t}$ with $\phi=u_r$ to obtain
\[
\int_{\Omega}\nabla u_{k+j,t}\cdot \nabla u_r\,dx
=\lambda_{k+j}(\rho_t)\int_{\partial\Omega}\rho_t\,u_{k+j,t}\,u_r\,dS.
\]
Similarly, test the weak formulation at $\rho$ for $u_r$ with $\phi=u_{k+j,t}$:
\[
\int_{\Omega}\nabla u_r\cdot \nabla u_{k+j,t}\,dx
=\lambda\int_{\partial\Omega}\rho\,u_r\,u_{k+j,t}\,dS.
\]
Since the left-hand sides are equal, subtracting and using $\rho_t=\rho+t h$ gives
\[
0=(\lambda_{k+j}(\rho_t)-\lambda)\int_{\partial\Omega}\rho\,u_{k+j,t}\,u_r\,dS
+\lambda_{k+j}(\rho_t)\,t\int_{\partial\Omega}h\,u_{k+j,t}\,u_r\,dS.
\]
Hence,
\begin{equation}\label{eq:key_identity}
\frac{\lambda_{k+j}(\rho_t)-\lambda}{t}\int_{\partial\Omega}\rho\,u_{k+j,t}\,u_r\,dS
=-\lambda_{k+j}(\rho_t)\int_{\partial\Omega}h\,u_{k+j,t}\,u_r\,dS.
\end{equation}

Define the $(l+1)\times(l+1)$ matrices
\[
C(t)_{rj}:=\int_{\partial\Omega}\rho\,u_{k+j,t}\,u_r\,dS,
\qquad
D(t)_{rj}:=\int_{\partial\Omega}h\,u_{k+j,t}\,u_r\,dS,
\]
and the diagonal matrices
\[
E(t):=\mathrm{diag}\Big(\frac{\lambda_k(\rho_t)-\lambda}{t},\dots,
\frac{\lambda_{k+l}(\rho_t)-\lambda}{t}\Big),
\qquad
L(t):=\mathrm{diag}\big(\lambda_k(\rho_t),\dots,\lambda_{k+l}(\rho_t)\big).
\]
Then \eqref{eq:key_identity} for all $r,j$ is exactly the matrix identity
\[
C(t)\,E(t)=-D(t)\,L(t).
\]
For $|t|$ small, $C(t)$ is invertible (indeed $C(t)$ converges to an orthogonal matrix, see below),
hence
\[
E(t)=-C(t)^{-1}D(t)L(t).
\]
Taking the trace yields
\begin{equation}\label{eq:trace_step}
\sum_{j=0}^l \frac{\lambda_{k+j}(\rho_t)-\lambda}{t}
=\mathrm{tr}\,E(t)
=-\mathrm{tr}\!\big(C(t)^{-1}D(t)L(t)\big).
\end{equation}

We now pass to the limit $t\to0$. By part (i) (applied to a sequence $t_i\to0$),
after extracting a subsequence we may assume
\[
\lambda_{k+j}(\rho_{t_i})\to\lambda,
\qquad
u_{k+j,t_i}\to v_j \ \text{in }L^2(\partial\Omega),
\]
where $\{v_j\}_{j=0}^l$ is a $\rho$-orthonormal basis of the eigenspace $E_\lambda$.
Consequently, there exists an orthogonal matrix $Q\in\mathbb{R}^{(l+1)\times(l+1)}$ such that
\[
v_j=\sum_{s=0}^l Q_{js}\,u_s,\qquad j=0,\dots,l.
\]
Define the symmetric matrix $M\in\mathbb{R}^{(l+1)\times(l+1)}$ by
\[
M_{rs}:=\int_{\partial\Omega}h\,u_r u_s\,dS.
\]
Using $u_{k+j,t_i}\to v_j$, we have the limits
\[
C(t_i)_{rj}\to \int_{\partial\Omega}\rho\,v_j u_r\,dS = Q_{jr},
\qquad
D(t_i)_{rj}\to \int_{\partial\Omega}h\,v_j u_r\,dS
=\sum_{s=0}^l Q_{js}M_{rs}.
\]
In matrix form,
\[
C(t_i)\to Q^\top,
\qquad
D(t_i)\to M\,Q^\top,
\qquad
L(t_i)\to \lambda I.
\]
Therefore,
\[
C(t_i)^{-1}D(t_i)L(t_i)\to (Q^\top)^{-1}(M Q^\top)(\lambda I)
=\lambda\,Q M Q^\top.
\]
Taking traces and using $\mathrm{tr}(QMQ^\top)=\mathrm{tr}(M)$ yields
\[
\lim_{i\to\infty}\mathrm{tr}\!\big(C(t_i)^{-1}D(t_i)L(t_i)\big)
=\lambda\,\mathrm{tr}(M)
=\lambda\sum_{r=0}^l\int_{\partial\Omega}h\,u_r^2\,dS.
\]
Combining with \eqref{eq:trace_step}, we conclude
\[
\lim_{t\to0}\sum_{j=0}^l \frac{\lambda_{k+j}(\rho_t)-\lambda}{t}
=-\lambda\sum_{r=0}^l\int_{\partial\Omega}h\,u_r^2\,dS
=-\lambda\int_{\partial\Omega}h\Big(\sum_{r=0}^l u_{k+r,\rho}^2\Big)\,dS,
\]
which proves that $\Lambda_{k,l}(\rho)=\sum_{j=0}^l\lambda_{k+j}(\rho)$ is G\^ateaux differentiable at $\rho$
and
\[
(\Lambda'_{k,l}(\rho),h)
= -\,\lambda_k(\rho)\int_{\partial\Omega} h\,\Big(\sum_{j=0}^l u_{k+j,\rho}^2\Big)\,dS.
\]
\emph{(iii)} (\emph{Fr\'echet differentiability})
Assume $N\le 3$ and fix $\{u_{k+j,\rho}\}_{j=0}^l$ as the $\rho$-orthonormal basis of $E_\lambda$. Let $\{\rho_i\}_1^\infty \subset U$ converges to $\rho \in U$ in $L^2(\partial\Omega)$.
By part~(i), after passing to a subsequence, the corresponding cluster eigenvalues satisfy
$\lambda_{k+j}(\rho_i)\to\lambda_{k+j}(\rho)$ for $j=0,\dots,l$, and there exist
$\rho_i$-orthonormal families of eigenfunctions
$\{u_{k+j,\rho_i}\}_{j=0}^l$ such that
\begin{equation}\label{eq:ukjConvTovj}
    u_{k+j,\rho_i}\rightharpoonup v_j \ \text{ in } H^1(\Omega),
\qquad
u_{k+j,\rho_i}\to v_j \ \text{ in } L^2(\partial\Omega),
\end{equation}
where $\{v_0,\dots,v_l\}$ is a $\rho$-orthonormal basis of $E_\lambda$. Since $N\le3$, Lemma~\ref{lem:requk} implies that $\{u_{k+j,\rho_i}\}$ is uniformly bounded in
$C(\overline{\Omega})$, hence in $L^\infty(\partial\Omega)$. In view of \eqref{eq:ukjConvTovj} and the uniform boundedness, we can invoke 
 the dominated convergence theorem, and   obtain
\[
\sum_{j=0}^l u_{k+j,\rho_i}^2 \to \sum_{j=0}^l v_{j}^2
\quad\text{in }L^2(\partial\Omega).
\]
We will show that $ \sum_{j=0}^l v_{j}^2 = \sum_{j=0}^l u_{k+j,\rho}^2$ and
together with $\lambda_k(\rho_i)\to\lambda_k(\rho)$, this shows that
$\Lambda'_{k,l}(\rho_i)\to\Lambda'_{k,l}(\rho)$ in $L^2(\partial\Omega)$.
Therefore, $\Lambda_{k,l}$ is Fr\'echet differentiable
with derivative given in~(ii), see, for instance, \cite[Proposition 5.3.4]{atkinson2009theoretical}.

Now,we show that
\[
\sum_{j=0}^l v_j^2 = \sum_{j=0}^l u_{k+j,\rho}^2.
\]
Indeed, since $\{v_0,\dots,v_l\}$ and $\{u_{k+j,\rho}\}_{j=0}^l$ are two
$\rho$-orthonormal bases of the same eigenspace $E_\lambda$, there exists an
orthogonal matrix $Q$ such that
\[
v_j=\sum_{s=0}^l Q_{js}\,u_{k+s,\rho},\qquad j=0,\dots,l.
\]
Consequently,
\[
\sum_{j=0}^l v_j^2
=\sum_{j=0}^l\Big(\sum_{s=0}^l Q_{js}u_{k+s,\rho}\Big)^2
=\sum_{s=0}^l u_{k+s,\rho}^2,
\]
where we used $Q^\top Q=I$.

\end{proof}

\begin{rem}
The enlargement of the admissible bounds by $\varepsilon>0$ is introduced to provide a neighbourhood of admissible perturbations around a given density $\rho \in \mathcal{M}$, while preserving uniform positivity. This allows one to consider variations $\rho+h$ that remain in the convex set $U$ for $\|h\|_{L^\infty(\partial \Omega)}$ sufficiently small. In particular, Fr\'echet differentiability in part~(iii) is understood \emph{relative to $U$}, i.e.\ the limit is taken over perturbations $h$ such that $\rho+h\in U$.
\end{rem}
\begin{rem}\label{rem:simpleeigs}
If $\lambda_k(\rho)$ is simple (i.e.\ $l=0$), then Lemma~\ref{lem:weakcon} reduces to the classical differentiability result for simple Steklov eigenvalues. In this case, the eigenfunction $u_{k,\rho}$ is uniquely determined up to a sign by the normalization
$\int_{\partial\Omega}\rho\,u_{k,\rho}^2\,dS=1$, and the G\^ateaux derivative in \emph{(ii)} simplifies to
\[
(\lambda_k'(\rho),h)
= -\,\lambda_k(\rho)\int_{\partial\Omega} h\,u_{k,\rho}^2\,dS.
\]
\end{rem}

Later on in this paper, we require the following technical lemma to derive the optimality condition. The lemma is closely related to the bathtub principle (see \cite[Theorem 1.14]{lieb2001analysis}). For completeness, we include the proof in the Appendix B.
\begin{lem}[Bathtub principle for extremal problems]\label{lem:bathtublem}
Let $f$ be a function in $L^1(\partial \Omega)$. 
Consider the following extremal problems
\[ \underset{\rho \in \mathcal{M}}{\min}\int_{\partial \Omega} \rho f  dS, \quad \text{with} \quad\tau = \inf \left\{ s \in \mathbb{R} : \left| \left\{ \mathbf{x} \in \partial \Omega : f(\mathbf{x}) \leq s \right\} \right| \geq A \right\},\]
\[ \underset{\rho \in \mathcal{M}}{\max}\int_{\partial \Omega} \rho f  dS,\quad \text{with} \quad\tau = \sup \left\{ s \in \mathbb{R} : \left| \left\{ \mathbf{x} \in \partial \Omega : f(\mathbf{x}) \geq s \right\} \right| \geq A \right\},\]
where $A$ is given by \eqref{eq:Aformula}.\\
i) Both problems admit  a solution. \\
ii) Function $\bar \rho\in \mathcal{M}$ is optimizer if and only if, almost everywhere on $\mathbf{x} \in \partial \Omega$, 
\begin{align*}   
\bar \rho(\mathbf{x}) = 
\begin{cases} 
   \beta,  & \text{if }\: f_{}(\mathbf{x}) < \tau, \\
   \in[\alpha,\beta], & \text{if }\: f_{}(\mathbf{x}) = \tau,\\
   \alpha,  & \text{if} \: f_{}(\mathbf{x})>\tau,
\end{cases} \qquad (\text{for minimization})
\end{align*}

\begin{align*}   
\bar \rho(\mathbf{x}) = 
\begin{cases} 
   \beta,  & \text{if }\: f_{}(\mathbf{x}) > \tau, \\
   \in[\alpha,\beta], & \text{if }\: f_{}(\mathbf{x}) = \tau,\\
   \alpha,  & \text{if} \: f_{}(\mathbf{x})< \tau.
\end{cases} \qquad (\text{for maximization})
\end{align*}
iii) There is a bang-bang optimizer  $\bar \rho = \alpha + (\beta -\alpha)\chi_{ E}$ such that
\begin{equation*}
		\{\mathbf{x}\in \partial \Omega:\: f(\mathbf{x})<\tau\}\subset  E \subset \{\mathbf{x}\in \partial \Omega:\: f(\mathbf{x})\leq\tau\},\qquad(\text{for minimization})
	\end{equation*}
\begin{equation*}
		\{\mathbf{x}\in \partial \Omega:\: f(\mathbf{x})>\tau\}\subset  E \subset \{\mathbf{x}\in \partial \Omega:\: f(\mathbf{x})\geq\tau\}.\qquad(\text{for maximization})
	\end{equation*}
iv)  If the optimizer $\bar \rho$ is unique,  then $\bar \rho  = \alpha + (\beta -\alpha)\chi_{ E}$ where
 \[E = \{\mathbf{x}\in \partial \Omega:\: f(\mathbf{x}) {\leq}\tau\}, \qquad (\text{for minimization})\] 
  \[E = \{\mathbf{x}\in \partial \Omega:\: f(\mathbf{x}) {\geq}\tau\}. \qquad (\text{for maximization})\] 
 \end{lem}

\section{Convexity and Concavity Phenomena in Elliptic Eigenvalue Problems}\label{sec:convexity}

The convexity or concavity of the objective functional plays a crucial role in optimization, as it allows one to leverage well-established results and algorithms from convex optimization theory. Unfortunately, this property does not generally hold for eigenvalue problems. In particular, eigenvalues of elliptic operators are typically neither convex nor concave with respect to the coefficients. For example, let \(\mu_k(\rho)\) be the \(k\)th eigenvalue of
\(-\Delta u=\mu\,\rho\,u\) in \(\Omega\) with \(u=0\) on \(\partial\Omega\). We give an  example showing that \(1/\mu_k(\rho)\) is neither convex nor concave for \(k=2,3,4\) in Appendix~A. In general, $\mu_k(\rho)$ and ${1}/{\mu_k(\rho)}$ are neither convex nor concave. For more details, see Appendix~A.

Consider the Steklov problem \eqref{eq:mainpde} on the unit disk $\Omega$ centered at the origin. We introduce two weight functions 
\[
\rho_{1}(\theta)=\begin{cases}
	\beta & 0\le \theta<\pi,\\
	\alpha& \pi \le \theta < 2 \pi,
\end{cases}\quad\text{and}\quad\rho_{2}(\theta)=\begin{cases}
	\alpha & 0\le \theta<\pi,\\
	\beta& \pi \le \theta < 2 \pi,
\end{cases}
\]
and the convex combination, $\rho_{t}=t\rho_{1}+(1-t)\rho_{2}$ for $0\le t\le1.$ We have
\[
\rho_{t}(\theta)=\begin{cases}
	\alpha+t(\beta-\alpha) & \theta<\pi,\\
	\beta-t(\beta-\alpha) & \theta\ge\pi.
\end{cases}
\]
In Figure~\ref{fig:sigma_1_8}, the first eight Steklov eigenvalues of \( \rho_t \), i.e., \( \lambda_k(t) := \lambda_k(\rho_t) \) for \( t \in [0,1] \), with \( \alpha = 0.5 \), \( \beta = 1.5 \), $\gamma= 2\pi$ and \( k = 1,\dots,8 \), are plotted. The figure shows that along this path $\rho_t$, $\lambda_1(t)$ looks concave or $1/\lambda_1(t)$ looks convex but higher modes $\lambda_k(t)$ for $k \ge 3$ are neither concave or convex. Even though $\lambda_1$ and $\lambda_2$ shows concavity and convexity along the path $\rho_t$, respectively, but we will demonstrate its non-concavity and non-convexity in Remark~\ref{rem:convexityconclusion}, after discussing the non-uniqueness of maximizers in Remark~\ref{rem:distmaximizers}.
 In general, the functionals $\lambda_k(\rho)$ and ${1}/{\lambda_k(\rho)}$, $k\geq 1,$ are not convex and not concave on the admissible set $\mathcal{M}$.

The next propositions give a simple derivative test: if there are two distinct maximizers and the directional derivative at an endpoint is nonzero, then $\lambda_k$  is not concave; with two minimizers, not convex. Symmetry often gives non-unique optimizers, so these tests apply in many cases.

\begin{figure}[h!]
\centering\includegraphics[width=0.8\textwidth]{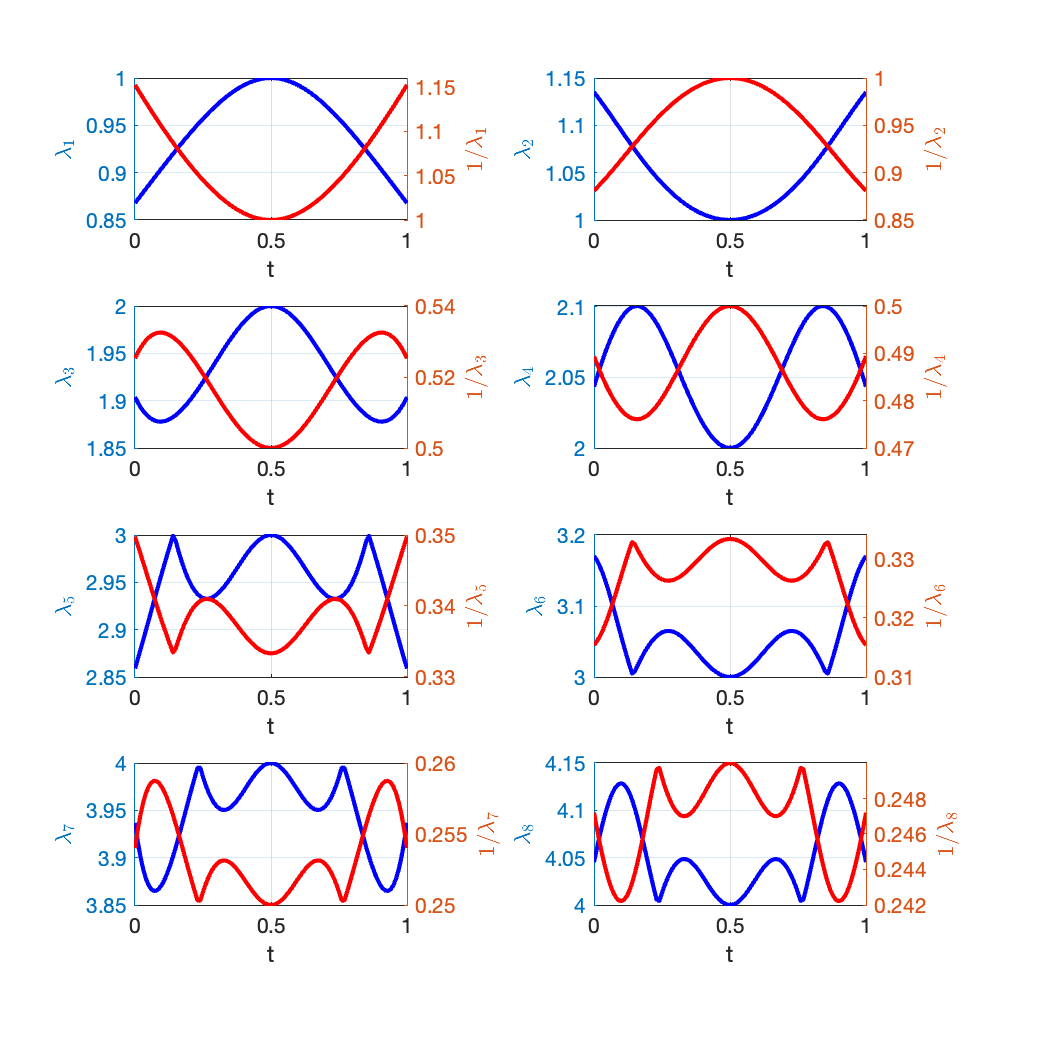}
\caption{The first eight Steklov eigenvalues $\lambda_k(\rho_t)$, $k=1,\cdots,8$ and their reciprocals $1/\lambda_k(\rho_t)$ with weight function $\rho_{t}=t\rho_{1}+(1-t)\rho_{2}$ with $\alpha = 0.5$, $\beta = 1.5$ and $\gamma = 2\pi$.\label{fig:sigma_1_8}}

\end{figure}

\begin{prop}\label{prop:nonconcavity-Lambda}
Fix $k\ge 1$. Assume the maximization problem \eqref{eq:maxp} admits
two distinct maximizers $\rho_0,\rho_1\in\mathcal M$. 
Suppose $\lambda_k(\rho_0)$ is simple and set $h:=\rho_1-\rho_0$.
If $\big(\lambda_k'(\rho_0),h\big)\neq 0$, then $\lambda_k$ is not concave on $\mathcal M$.
\end{prop}

\begin{proof}
Consider the segment $\rho_t=\rho_0+t h$, $t\in[0,1]$.
Due to the convexity of the set $\mathcal M$, $\rho_t \in \mathcal M$.
If $\lambda_k(\rho)$ were concave on the convex set $\mathcal M$ and both endpoints were maximizers, 
then $t\mapsto \lambda_k(\rho_t)$ would be constant on $[0,1]$, hence
\[
\frac{d}{dt}\lambda_k(\rho_t)\Big|_{t=0}=0.
\]
By differentiability at $\rho_0$  and  simplicity, Lemma \ref{lem:weakcon} yields 
\[
\frac{d}{dt}\lambda_k(\rho_t)\Big|_{t=0} \;=\; \big(\lambda_k'(\rho_0),h\big),
\]
which is nonzero by hypothesis, a contradiction.
\end{proof}
\begin{rem}[Non-convexity]
The same derivative argument shows non-convexity: if there exist two distinct minimizers 
$\rho_0,\rho_1\in\mathcal M$ of $\lambda_k(\rho)$ and $h=\rho_1-\rho_0$, then—at an endpoint where $\lambda_k(\rho_0)$ is simple—if 
$(\lambda_k'(\rho_0),h)\neq 0$, it follows that $\lambda_k(\rho)$ is not convex on $\mathcal M$. 
\end{rem}

\begin{rem}[Symmetry and non-uniqueness]
If the domain $\Omega$ and the admissible set $\mathcal M$ are invariant under boundary isometries $T$
(e.g., rotations/reflections on the unit disk), then for any optimizer $\hat\rho$ (minimizer or maximizer) of
$\lambda_k(\rho)$, the pushforward $\hat\rho\circ T$ is also an optimizer. 
If $\hat\rho\circ T\neq \hat\rho$ a.e.\ for some $T$, optimizers are not unique. 
This symmetry--induced non-uniqueness is a natural setting to apply the derivative tests above to rule out
concavity (from two maximizers) or convexity (from two minimizers) along segments joining distinct optimizers.
\end{rem}

\section{Minimization problem \eqref{eq:minp}}\label{sec:minprob}

This section is devoted to the study of problem \eqref{eq:minp}. At first, we prove the problem admits a bang-bang solution. 

\begin{thm}\label{thm:existmin}
There exists a bang--bang density $\hat\rho=\alpha+(\beta-\alpha)\chi_{\hat D}\in\mathcal M$
that solves \eqref{eq:minp} for every fixed $k\ge 1$.
\end{thm}

\begin{proof}
\emph{Step 1: Existence.}
By Lemma~\ref{lem:weakcon}(i), $\rho\mapsto \lambda_k(\rho)$ is weakly sequentially
continuous, and $\mathcal M$ is weak-* compact, so \eqref{eq:minp} has a minimizer
$\bar\rho\in\mathcal M$.

\smallskip
\emph{Step 2: Reciprocal min--max on the mean–zero space.}
Using the variational characterization for the reciprocal of the eigenvalues, see \cite[page 99]{bandle1980isoperimetric}, we have
\[
\frac{1}{\lambda_k(\rho)}
=\max_{\substack{V\subset H^1_{\rho,0}\\ \dim V=k}}
\ \min_{\substack{0\neq u\in V\\ \int_\Omega|\nabla u|^2\,dx=1}}
\ \int_{\partial\Omega}\rho\,u^2\,dS,
\]
where $V$ is a subspace of dimension $k$.

\smallskip
\emph{Step 3: Fix a $k$–dimensional subspace at the minimizer.}
Let $u_{1,\bar\rho},\dots,u_{k,\bar\rho}$ be the first $k$ positive Steklov eigenfunctions
for $\rho=\bar\rho$, normalized to be energy orthonormal:
\[
\int_\Omega \nabla u_{i,\bar\rho}\!\cdot\!\nabla u_{j,\bar\rho}\,dx=\delta_{ij},\quad i,j=1,\dots,k.
\]
Set $U:=\mathrm{span}\{u_{1,\bar\rho},\dots,u_{k,\bar\rho}\}$. For any $\rho$, define the
$k\times k$ matrix
\[
M(\rho):=\big(m_{ij}(\rho)\big)_{i,j=1}^k,\qquad
m_{ij}(\rho):=\int_{\partial\Omega}\rho\,u_{i,\bar\rho}\,u_{j,\bar\rho}\,dS.
\]
If, in addition, $U\subset H^1_{\rho,0}$ (i.e. $\int_{\partial\Omega}\rho\,u_{i,\bar\rho}\,dS=0$
for all $i$), then taking $V=U$ in Step~2 gives
\begin{equation}\label{eq:lower-1overlambda-new}
\frac{1}{\lambda_k(\rho)}\
\ge\ \min_{\substack{0\neq u\in U\\ \int_\Omega|\nabla u|^2=1}}
\int_{\partial\Omega}\rho\,u^2\,dS
\ =\ \lambda_{\min}\big(M(\rho)\big).
\end{equation}

\smallskip
\emph{Step 4: An affine slice where $U$ is admissible and $M(\rho)$ is fixed.}
Consider the convex, nonempty affine slice
\[
\widehat{\mathcal M}:=\Bigl\{\rho\in\mathcal M:\ 
\int_{\partial\Omega}\rho\,u_{i,\bar\rho}\,dS=0\ (i=1,\dots,k),\
\int_{\partial\Omega}\rho\,u_{i,\bar\rho}\,u_{j,\bar\rho}\,dS=\frac{\delta_{ij}}{\lambda_i(\bar\rho)}\ (i,j=1,\dots,k)\Bigr\}.
\]
Clearly $\bar\rho\in\widehat{\mathcal M}$, so the slice is nonempty. For any $\rho\in\widehat{\mathcal M}$,
we have $U\subset H^1_{\rho,0}$ and $M(\rho)=\mathrm{diag}\big(1/\lambda_1(\bar\rho),\dots,1/\lambda_k(\bar\rho)\big)$,
hence
\[
\lambda_{\min}\big(M(\rho)\big)=\frac{1}{\lambda_k(\bar\rho)}.
\]
Combining with \eqref{eq:lower-1overlambda-new} yields
\[
\frac{1}{\lambda_k(\rho)}\ \ge\ \frac{1}{\lambda_k(\bar\rho)}
\quad\Longrightarrow\quad
\lambda_k(\rho)\ \le\ \lambda_k(\bar\rho)
\qquad\text{for all }\rho\in\widehat{\mathcal M}.
\]
Thus every extreme point of {the convex set} $\widehat{\mathcal M}$ is a minimizer of \eqref{eq:minp}.

\smallskip
\emph{Step 5: Extreme points are bang--bang.}
The set $\widehat{\mathcal M}$ is defined by box constraints $\alpha\le \rho\le\beta$, a mass
constraint $\int_{\partial\Omega}\rho\,dS=\gamma$, and finitely many linear  constraints
on $\rho$. By standard bathtub/extreme–point arguments (see, e.g., \cite[Lemma~2]{friedland1977extremal}),
its extreme points are bang--bang: $\rho=\alpha+(\beta-\alpha)\chi_D$ with $|D|=A$.
Choosing such an extreme point $\hat\rho\in\widehat{\mathcal M}$ gives a bang--bang minimizer.
\end{proof}

Next, we provide an optimality condition for a solution of \eqref{eq:minp}.
\begin{thm}\label{thm:OCmin}
Let \( \hat{\rho} \) be a minimizer of \( \lambda_k(\rho) \) over \( \rho \in \mathcal{M} \), and assume that \( \lambda_k(\hat{\rho}) \) has multiplicity \( l+1 \), where \( l \geq 0 \). Then
\[
\hat{\rho}(\mathbf{x}) = 
\begin{cases} 
\alpha, & \text{if } w(\mathbf{x}) < \tau, \\
\in [\alpha, \beta], & \text{if } w(\mathbf{x}) = \tau, \\
\beta, & \text{if } w(\mathbf{x}) > \tau,
\end{cases}
\]
where \( w(\mathbf{x}) = \sum_{j=0}^{l} u_{k+j,\hat{\rho}}^2(\mathbf{x}) \), and
$
\tau = \sup \left\{ s \in \mathbb{R} : \left| \left\{ \mathbf{x} \in \partial \Omega : w(\mathbf{x}) \geq s \right\} \right| \geq A \right\}.
$
\end{thm}
\begin{proof}
   It is straightforward to verify that if \( \hat{\rho} \) is a minimizer of \( \lambda_k(\rho) \), then it also minimizes the functional \( \Lambda_{k,l}(\rho) \) defined in Lemma~\ref{lem:weakcon}. According to Lemma~\ref{lem:weakcon}, this functional is differentiable. Moreover, it is standard (see, for example, \cite[Lemma 2.21]{troltzsch2010optimal}) that any minimizer \( \hat{\rho} \in \mathcal{M} \) satisfies the following variational inequality:
\[
\left(\Lambda_{k,l}^\prime(\hat\rho), \rho - \hat\rho \right) = -\lambda_k(\hat\rho) \int_{\partial\Omega} (\rho - \hat\rho) \, w \, dS \geq 0, \quad \text{for all } \rho \in \mathcal{M}.
\]
This implies that \( \hat{\rho} \) maximizes the functional
$\int_{\partial\Omega} \rho \, w \, dS$ {over} $ \rho \in \mathcal{M}.$
The desired assertion then follows by invoking Lemma~\ref{lem:bathtublem}-(ii).
\end{proof}

The optimality condition offers insights into the topology and geometry of the optimizer, guiding our numerical investigations. By continuity of eigenfunctions (Lemma~\ref{lem:requk}) and the identity \( \int_{\partial \Omega} \rho\, u_{k,\rho}\, dS = 0 \) for \( k \geq 1 \), all nontrivial Steklov eigenfunctions change sign and vanish somewhere on \( \partial \Omega \). Since the number and structure of these zeros remain unknown \cite{colbois2024some,Girouard_2017}, we make assumptions on their number when needed.

\begin{cor}\label{cor:disconnectedOptimSet}
Let \(\Omega \subset \mathbb{R}^2\) be a planar domain, and let \(\hat \rho = \alpha + (\beta - \alpha)\chi_{\hat D}\) be a minimizer of the simple eigenvalue \(\lambda_k(\rho)\) for some \(k \geq 1\). Assume the associated eigenfunction \(u_{k,\hat \rho}\) has \(m \geq 2\) distinct zeros on \(\partial \Omega\), and that \(A\) is sufficiently large. Then \(\hat D \subset \partial \Omega\) is disconnected in the relative topology of \(\partial \Omega\), with at least \(m\) connected components.
\end{cor}

\begin{proof}
Parametrize the boundary \(\partial \Omega\) by arc length \(s \in [0, |\partial \Omega|]\), and define \(w(s) := u_{k,\hat \rho}^2(s)\). Without loss of generality, suppose that \(w(0) = w(|\partial \Omega|) = 0\), meaning the boundary of \( \partial \Omega \) starts and ends at  a point where \( u_{k,\hat \rho} = 0 \). Since \(u_{k,\hat \rho}\) has \(m\) distinct zeros, there exist \(s_1 < \cdots < s_{m-1} \in (0, |\partial \Omega|)\) such that \(w(s_i) = 0\) for \(i = 1, \dots, m - 1\).

By the optimality condition in Theorem~\ref{thm:OCmin}, we have
\[
\hat D \subset \left\{ s \in (0, |\partial \Omega|) : w(s) \geq \tau \right\}.
\]
If \(A\) is large enough, then around each zero \(s_i\), there exists a nontrivial open arc \(I_i\) where \(w(s) < \tau\), and hence \(I_i \cap \hat D = \emptyset\). Since these intervals are disjoint and separate the support of \(\hat D\), it follows that \(\hat D\) must have at least \(m\) connected components.
\end{proof}

\section{Maximization problem \eqref{eq:maxp}}\label{sec:maxprob}

This section is devoted to the maximization problem \eqref{eq:maxp}.  
We begin by addressing the question of existence and identifying the optimal solution in the case where \(\Omega\) is a circle in \(\mathbb{R}^2\). 
It is observed that, in contrast to the minimization problem, the maximizer is not necessarily of bang-bang type.

\begin{thm}\label{thm:MaxExistAnalSol}
The optimization problem \eqref{eq:maxp} admits a solution \( \check{\rho} \). Moreover, if \( \Omega \) is a circle, then \( \check{\rho}(\mathbf{x}) \equiv \gamma / |\partial \Omega| \) is a maximizer for \eqref{eq:maxp} when $k=1$.
\end{thm}

\begin{proof}
Existence of a maximizer \( \check{\rho} \in \mathcal{M} \) follows by an argument similar to that used for the minimization problem.

Let us denote the first Steklov eigenvalue as \( \lambda_1(\rho, \Omega) \) to emphasize its dependence on both \( \rho \) and \( \Omega \).
To identify the maximizer in the case when \( \Omega \subset \mathbb{R}^2 \) is a circle, we use a classical result by Weinstock~\cite{Weinstock_1954} and \cite[Theorem 3.1]{colbois2024some}. It states that among all simply connected planar domains and for densities \( \rho \) with fixed mass \( \int_{\partial \Omega} \rho \, ds = \gamma \), the first Steklov eigenvalue satisfies \( \lambda_1(\rho,\Omega) \leq 2\pi/\gamma \), with equality attained if  \( \Omega \) is a disk and \( \rho \) is constant. Using this, we obtain
\[
\max_{\rho \in \mathcal{M}} \lambda_1(\rho)\gamma \leq \max_{(\rho,\Omega) \in \mathcal{M}_{\rho,\Omega}} \lambda_1(\rho, \Omega)\gamma \leq 2\pi,
\]
where
\[
\mathcal{M}_{\rho,\Omega} := \left\{ (\rho, \Omega) : \rho(\mathbf x)>0,\:\int_{\partial \Omega} \rho \, ds = \gamma,\ \Omega \subset \mathbb{R}^2\ \text{simply connected} \right\}.
\]
Therefore, the optimal value \( 2\pi \) is achieved within this larger class when \( \Omega \) is a disk and \( \rho(\mathbf{x}) \equiv \gamma / |\partial \Omega| \). In particular, fixing \( \Omega \) as a circle, the same constant density \( \rho(\mathbf{x}) \equiv \gamma / |\partial \Omega| \) achieves the maximum in the admissible class \( \mathcal{M} \). 
\end{proof}

\begin{rem}[Non-uniqueness of maximizers on the disk]\label{rem:distmaximizers}
In the minimization problem over a circular domain, the existence of infinitely many minimizers is explained by rotational symmetry, as any rotation of a given minimizer yields another. By contrast, in the maximization problem, in particular for maximizing $\lambda_1$, we establish the existence of infinitely many maximizers whose distinct profiles are not consequences of domain symmetries.

As it is well known that all biholomorphic maps of the unit disk $D$ in the complex plane
to itself are the M\"{o}bius transformations of the form
\[
f(z)=e^{i\varphi}\frac{z-a}{1-\bar{a}z},
\]
where $z$ is the complex variable, $a\in D$ is a point inside the unit circle and $\varphi\in[0,2\pi)$.
The derivative of $f(z)$ is given by 
\[
f'(z)=e^{i\varphi}\frac{1-|a|^{2}}{(1-\bar{a}z)^{2}}.
\]
We then have 
\[
\rho=|f'(z)|=\left|\frac{1-|a|^{2}}{(1-\bar{a}z)^{2}}\right|.
\]
Furthermore, let $z=e^{i\theta}$ and $a=re^{i\eta}$ with $r<1$,
\begin{equation*}
    \int_{\partial D}\rho ds =\int_{|z|=1}\left|\frac{1-|a|^{2}}{(1-\bar{a}z)^{2}}\right||dz| =\int_{0}^{2\pi}\frac{1-r^{2}}{1-2r\cos(\theta-\eta)+r^{2}}d\theta =2\pi,
\end{equation*}
where we apply the Poisson's integral  formula. Recall that it was proved by Weinstock in 1954~\cite{Weinstock_1954} that, among simply connected planar domains and for densities \( \rho \) with fixed mass \( \int_{\partial \Omega} \rho \, dS = \gamma \), the first Steklov eigenvalue satisfies \( \lambda_{1}(\rho) \leq 2\pi / \gamma \), with equality attained when \( \Omega \) is a disk, and \( \rho \) is constant.
Hence, when $\rho \equiv 1$, we know that the first eigenvalue reaches its maximum value $\lambda_1=\lambda_2=1$ with the corresponding eigenfunctions $r\cos(\theta)$ and $r\sin(\theta)$. Consider the conformal mapping $f(z)$ from unit disk to itself. Define $v = u \circ f^{-1}$. The function $v$ satisfies the non-weighted Steklov eigenvalue problem  \begin{equation}\label{eq:non-weightedmainpde}
\begin{cases}
\Delta v=0 & \text{in } \Omega,\\
\displaystyle \frac{\partial v}{\partial \mathbf n}=\lambda \, \,v & \text{on } \partial\Omega,
\end{cases}
\end{equation}
so $\max \lambda_1 \int_{\partial D} \rho ds = 2\pi$ as well. This means that each conformal map $f$ of the disk to itself yields a density $\rho=|f'|$ that also attains the maximum $\lambda_1\gamma=2\pi$. 
Thus, the maximization problem on the disk admits infinitely many maximizers, but these are not
generated by rotational symmetry. Instead, they arise from the rich family of conformal self-maps
of the unit disk and so have distinct profiles. In Figure \ref{fig:Disk max lambda1}, we plot $\rho$ for $a=0,0.1,0.3,$ and $0.5$ to illustrate
that a few different $\rho$ functions with distinct profiles which all have $\lambda_{1}\gamma=2\pi.$ 
\end{rem}

\begin{rem}\label{rem:convexityconclusion}
   Let $\rho_1$ be the density function with $a=0$ and $\rho_2$ the one with $a=0.5$, as defined in Remark~\ref{rem:distmaximizers}.
 Consider $\rho_t$ is the convex combination, i.e., $\rho_{t}=t\rho_{1}+(1-t)\rho_{2}$ for $0\le t\le1.$ In Figure \ref{fig:Nconv_lambda1_Nconvex_lambda2}, we observe that $\lambda_1(\rho_t)$ is not concave and $\lambda_2(\rho_t)$ is not convex anymore, comparing to Figure \ref{fig:sigma_1_8}. Moreover, we observe that $1/\lambda_1(\rho_t)$ is not convex $1/\lambda_2(\rho_t)$ is not concave, comparing to Figure \ref{fig:sigma_1_8}. We conclude that, in general, the functionals $\lambda_k(\rho)$ and ${1}/{\lambda_k(\rho)},$ $k\geq 1,$ are neither convex nor concave on the admissible set $\mathcal{M}$.
\end{rem}

\begin{figure}
	\centering
	\includegraphics[scale=1]{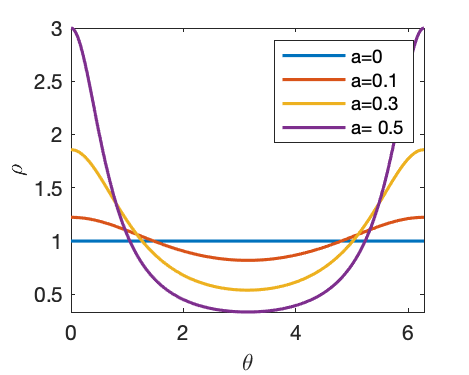}\caption{\label{fig:Disk max lambda1} The plots of several optimal $\rho$ for maximizing $\lambda_{1} \gamma$ on a disk {with $\alpha= 0.5. \beta= 1.5$ and $\gamma=2\pi.$}}	
\end{figure}

\begin{figure}
	\centering
	\includegraphics[scale=0.75]{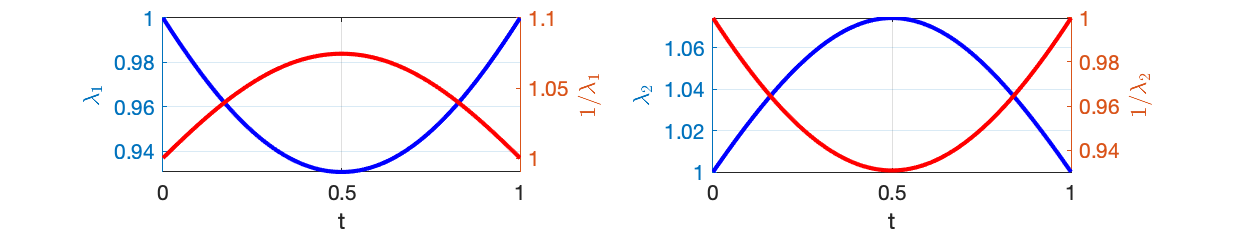}\caption{\label{fig:Nconv_lambda1_Nconvex_lambda2} The plots of first two eigenvalues of $\rho_{t}=t\rho_{1}+(1-t)\rho_{2}$ for $0\le t\le1$ where $\rho_1$ is the maximizer of $\lambda_1$ with $a=0$ and $\rho_2$ is the maximizer of $\lambda_1$ with $a=0.5$.}	
\end{figure}

Next, we reformulate the maximization problem to facilitate the derivation of optimality conditions. Instead of the original problem \eqref{eq:maxp}, we consider the modified objective
\begin{equation}\label{eq:maxpM}
   \max_{\rho \in \mathcal{M}} \Lambda_k(\rho), 
\end{equation}
where \( \Lambda_k(\rho) \) is defined as in  Lemma \ref{lem:weakcon}. This replacement is justified by the fact that \( \Lambda_k \) is Fréchet differentiable, which enables us to derive necessary optimality conditions and develop numerical methods. Moreover, when \( \lambda_k(\rho) \) is simple, we have \( \Lambda_k(\rho) = \lambda_k(\rho) \). Most importantly, It is easy to check that  any maximizer of the modified problem \eqref{eq:maxpM} is also a maximizer of the original problem \eqref{eq:maxp}.

Similar to the reasoning used for the minimization problem \eqref{eq:minp}, we can establish the following theorem, whose proof is omitted for brevity.
\begin{thm}\label{thm:OCmaxandexist}
Let $\check \rho$ be a maximizer of \eqref{eq:maxpM} and so  be a maximizer of \( \lambda_k(\rho) \) over \( \rho \in \mathcal{M} \), and assume that \( \lambda_k(\check{\rho}) \) has multiplicity \( l+1 \), where \( l \geq 0 \). Then
\[
\check{\rho}(\mathbf{x}) = 
\begin{cases} 
\beta, & \text{if } w(\mathbf{x}) < \tau, \\
\in [\alpha, \beta], & \text{if } w(\mathbf{x}) = \tau, \\
\alpha, & \text{if } w(\mathbf{x}) > \tau,
\end{cases}
\]
where \( w(\mathbf{x}) = \sum_{j=0}^{l} u_{k+j,\check{\rho}}^2(\mathbf{x}) \), and
$
\tau = \inf \left\{ s \in \mathbb{R} : \left| \left\{ \mathbf{x} \in \partial \Omega : w(\mathbf{x}) \leq s \right\} \right| \geq A \right\}.
$
\end{thm}

\section{Numerical Algorithms and Results}\label{sec:numeric}
This section is devoted to solving the optimization problems \eqref{eq:minp}–\eqref{eq:maxp} numerically. At first, we need a numerical method to discretize the elliptic eigenvalue problem \eqref{eq:mainpde}.

To solve the weighted Steklov eigenvalue problem \eqref{eq:mainpde} with given $\rho$, we use the method of particular solutions which was proposed in \cite{oudet2021computation,schroeder2023steklov}. When $\Omega$ is a simply connected domain, the approach is particularly simple. A harmonic Steklov eigenfunction $u$  can be approximated by a linear combination of harmonic functions in polar coordinates. We take
\[
\mathcal{B}
=\bigl\{\,\phi_1,\dots,\phi_{2k_0+1}\,\bigr\}
=\Bigl\{\,1,\ r^1\cos\theta,\ r^1\sin\theta,\ \dots,\ r^{k_0}\cos(k_0\theta),\ r^{k_0}\sin(k_0\theta)\Bigr\},
\]
and approximate $u$ by 
\[
u \approx\sum_{m=1}^{2k_0+1} V_m\,\phi_m,
\]
where $k_0$ is a chosen integer. We denote 
$(p_l)_{1 \leq l \leq L}$ as a uniform sampling with respect to arc 
length of $\partial \Omega$ with the step size $ds$. We approximate solutions 
of eigenvalue problem $(\lambda,u)$ by the solution $(\sigma_d,V)$ of the discrete generalized eigenvalue problem
\begin{equation} \label{e:mpseig}
B^T D_1 A  \ V =  \sigma_d  \ B^T D_2 B \  V, 
\end{equation}
where
\[
B=\bigl(\phi_m(p_l)\bigr)_{1\le l\le L,\ 1\le m\le 2k_0+1},
\qquad
A=\Bigl(\tfrac{\partial \phi_m}{\partial\mathbf n}(p_l)\Bigr)_{1\le l\le L,\ 1\le m\le 2k_0+1},
\]
and the diagonal  matrices
\[
D_1=\operatorname{diag}(ds,\dots,ds),\qquad
D_2=\operatorname{diag}\bigl(\rho(p_1)ds,\dots,\rho(p_L)ds\bigr).
\]

  Note that both $B^T D_1 A$ and $B^T D_2 B$ are both symmetric here. In the numerical implementation, we choose $k_0=18$ and $L = 720$.

As in \cite{Akhmetgaliyev2016,oudet2021computation}, to handle multiple eigenvalues, we transform $\min \lambda_{k} $ to $\min t$ with $t\ge \lambda_i$ for $i=1:k$ and  $\max \lambda_{k}$ to $\max t$ 
with $t\le \lambda_{i}$ for $i=k,k+1,\cdots, k+j-1$, respectively. A choice of small $j$ would work as long as it is greater or equal to the multiplicity of eigenvalue. In our simulation, $j=5$ is enough. The optimization is done by using fminimax in MATLAB with the gradient of eigenvalues provided.  This ensures the efficient and accurate calculation of the optimizers. 

In our numerical examples, we have considered $\Omega$ as the unit circle.
In Figure \ref{fig:Disk min}, the minimizers of $\lambda_{k}\gamma$ with $\alpha = 0.5$, $\beta = 1.5$, and $\gamma = 2\pi $ for $k=1:5$ are shown. The minimizers are bang-bang functions. The value of $\rho$ either takes $\alpha$ or $\beta$. When minimizer $k$-th Steklov eigenvalues, the region that $\rho = \beta$ consists of $(k+1)$ uniformly-distributed disconnected intervals. When $k$ is odd, the eigenvalue is simple while when $k$ is even, the eigenvalue has multiplicity two. We plot the corresponding eigenfunctions and the sum of their squares, $\sum u_{k,\rho}^2$, to verify the optimality conditions.  
In Figure \ref{fig:Disk min 2}, the minimizers of $\lambda_{k}\gamma$ with $\alpha = 0.25$, $\beta = 4$, and $\gamma = 2\pi $ for $k=1:5$ are shown. Again, the minimizers of  $\lambda_{k}\gamma$ are periodic bang-bang functions and high-density intervals $\{x|\rho(x) = \beta\}$ are disconnected and uniformly distributed along the boundary.

\begin{figure}
	\centering
	\includegraphics[scale=0.75]{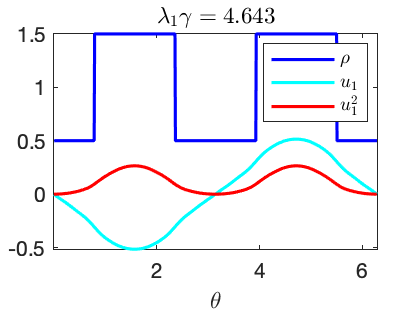}\includegraphics[scale=0.75]{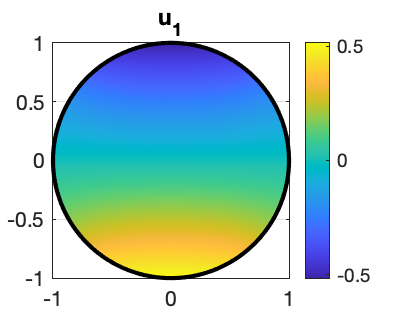}\\
    \includegraphics[scale=0.75]{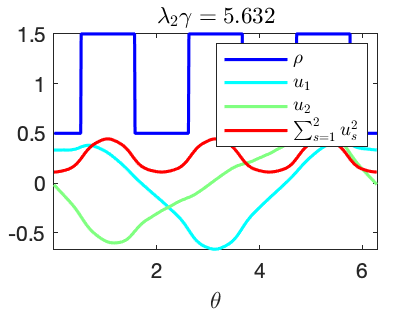}\includegraphics[scale=0.75]{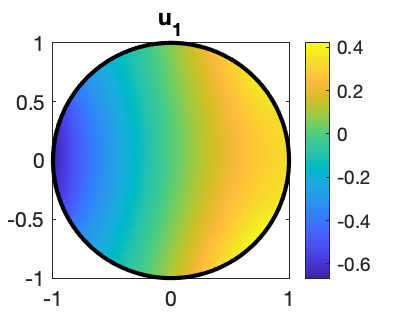}\includegraphics[scale=0.75]{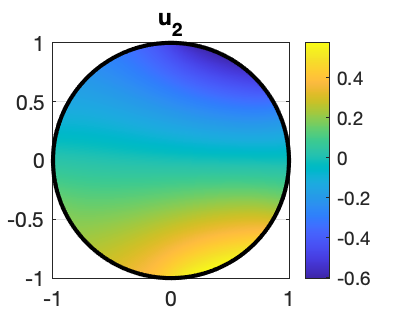}\\
 \includegraphics[scale=0.75]{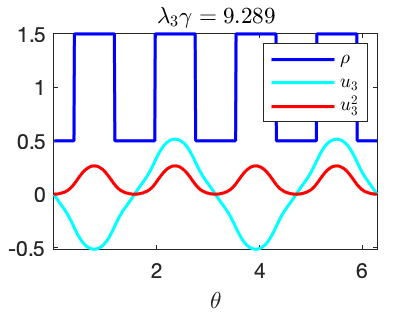}\includegraphics[scale=0.75]{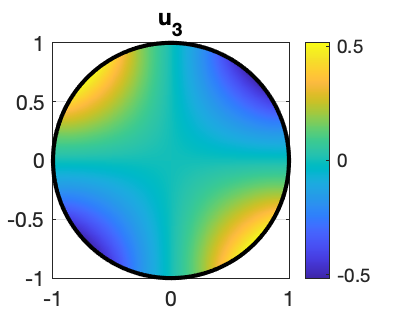}\\
    \includegraphics[scale=0.75]{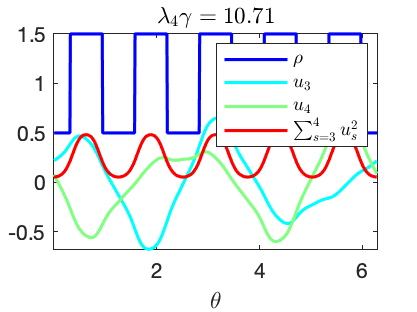}\includegraphics[scale=0.75]{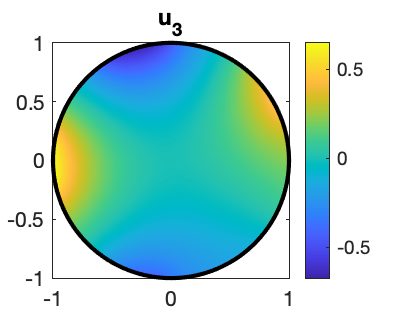}\includegraphics[scale=0.75]{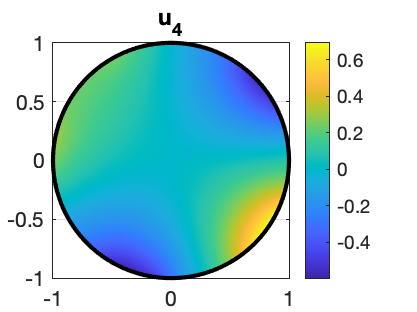}\\    
    \includegraphics[scale=0.75]{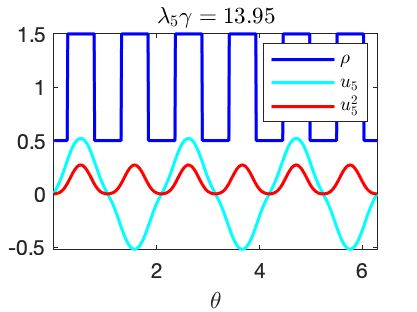}\includegraphics[scale=0.75]{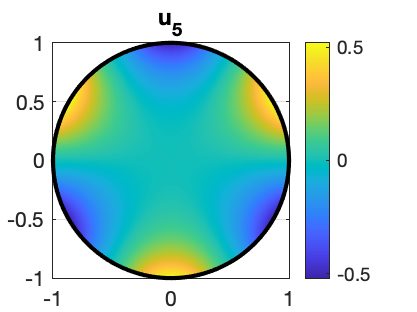}\\
    \caption{\label{fig:Disk min} The plots of optimal $\rho$ for minimizing $\lambda_{k} \gamma$ on a disk for $k=1:5$ with $\alpha = 0.5, \beta = 1.5$, and $\gamma = 2\pi.$  The corresponding eigenfunctions and the sum of their squares  are plotted to check the optimality conditions.}	
\end{figure}

\begin{figure}
	\centering
	\includegraphics[scale=0.75]{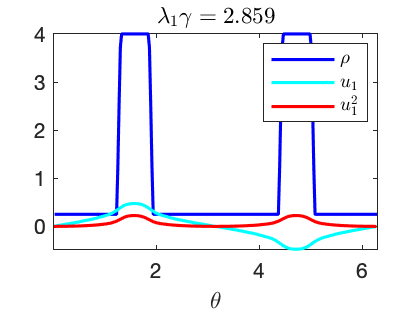}\includegraphics[scale=0.75]{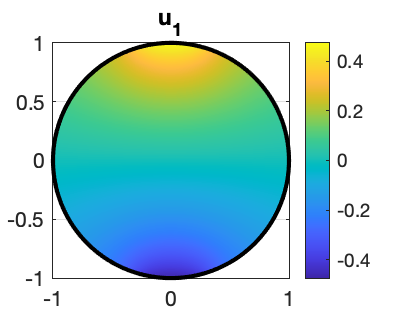}\\
    \includegraphics[scale=0.75]{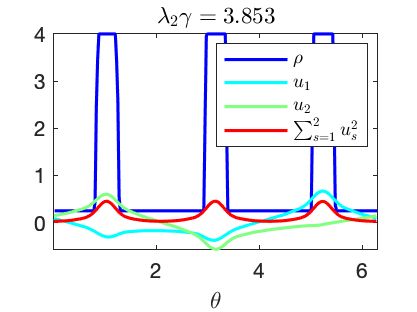}\includegraphics[scale=0.75]{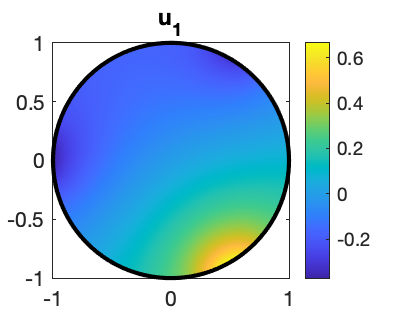}\includegraphics[scale=0.75]{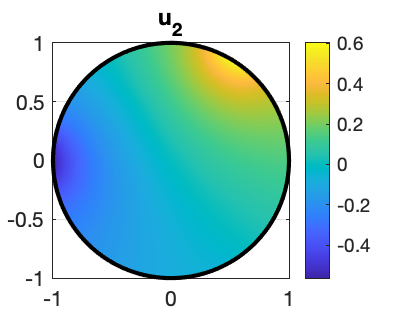}\\
 \includegraphics[scale=0.75]{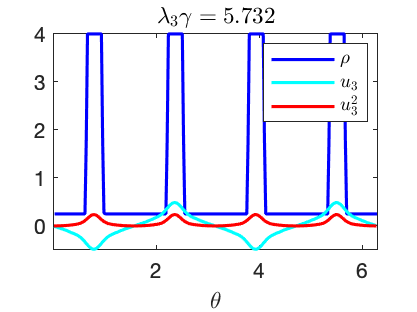}\includegraphics[scale=0.75]{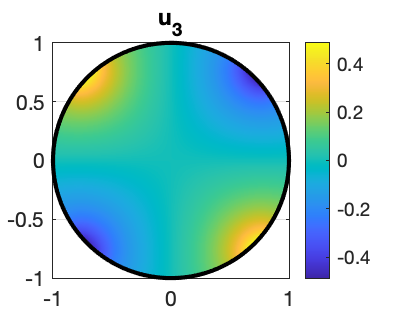}\\
    \includegraphics[scale=0.75]{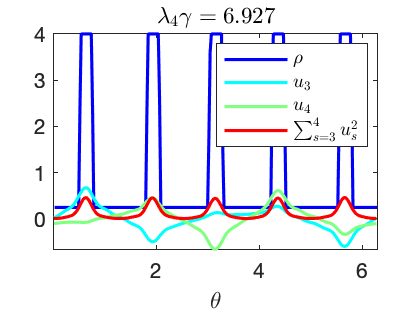}\includegraphics[scale=0.75]{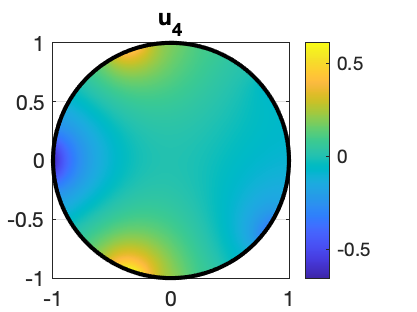}\includegraphics[scale=0.75]{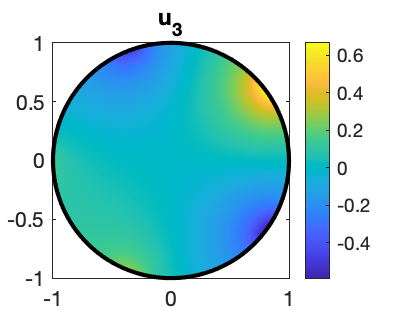}\\    
    \includegraphics[scale=0.75]{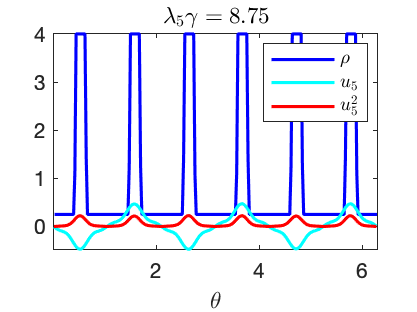}\includegraphics[scale=0.75]{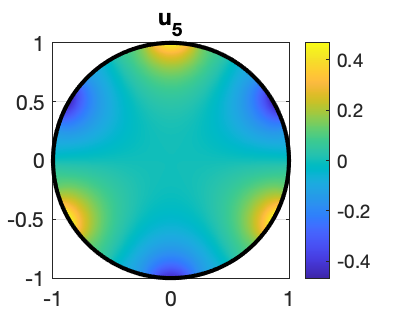}\\
    \caption{\label{fig:Disk min 2} The plots of optimal $\rho$ for minimizing $\lambda_{k} \gamma$ on a disk for $k=1:5$ with $\alpha = 0.25, \beta = 4$, and $\gamma = 2\pi.$  The corresponding eigenfunctions and the sum of their squares  are plotted to check the optimality conditions.}	
\end{figure}

\begin{figure}
	\centering
	\includegraphics[scale=0.75]{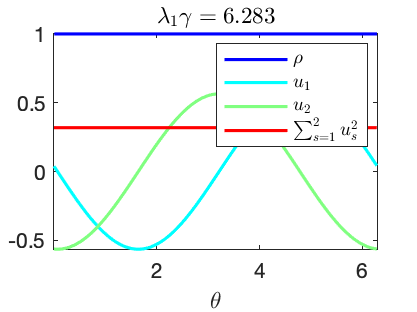}\includegraphics[scale=0.75]{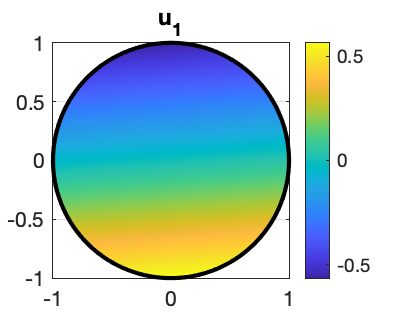}\includegraphics[scale=0.75]{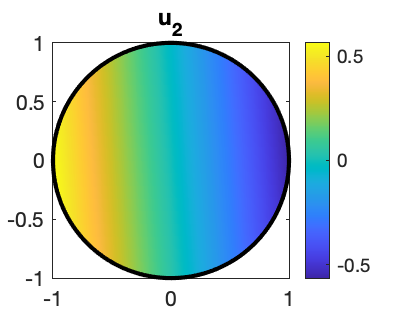}\\
  	\includegraphics[scale=0.75]{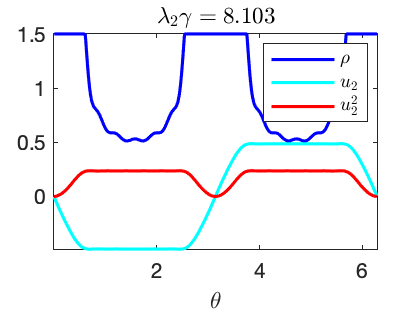}\includegraphics[scale=0.75]{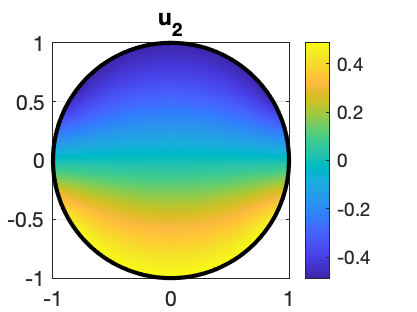}\\
\includegraphics[scale=0.75]{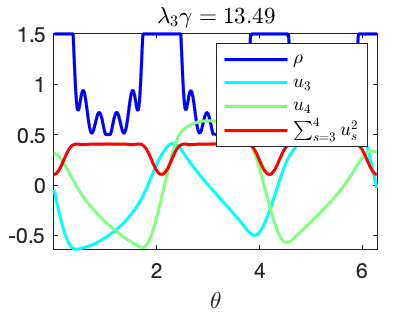}\includegraphics[scale=0.75]{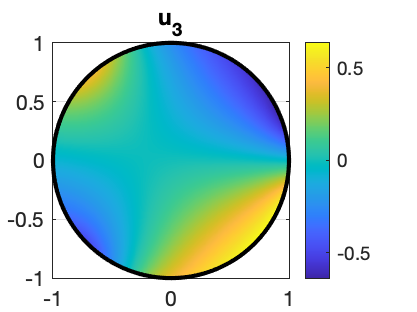}\includegraphics[scale=0.75]{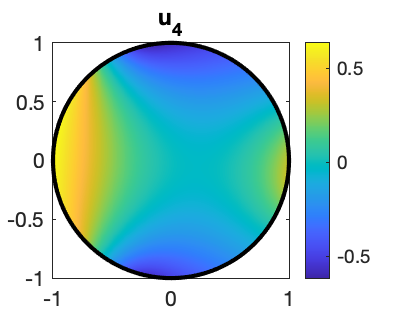}\\  
  \caption{\label{fig:Disk max} The plots of optimal $\rho$ for maximizing $\lambda_{k} \gamma$ on a disk for $k=1:3$ with $\alpha = 0.5, \beta = 1.5$, and $\gamma = 2\pi.$ The corresponding eigenfunctions and the sum of their squares are plotted to check the optimality conditions.}	
\end{figure}

In Figure \ref{fig:Disk max}, we show the maximizing results for $k=1:3$. The maximal value of $\lambda_1 \gamma$ is achieved by a constant density $\rho(\mathbf{x}) \equiv 1$. For higher eigenvalues, the optimizers are oscillatory so we only show the optimizers for maximizing $\lambda_2$ and $\lambda_3$ that the oscillations are resolved by the number of grids. We expect this oscillatory behavior. It is known that maximizing unweighted Steklov eigenvalues $\lambda_k(\rho(\mathbf{x}) \equiv 1)$ 
with respect to a simply connected domain $\Omega$ are  by the degenerating $k$-kissing disks. If one considers the conformal mapping from the $k$-kissing disks to a disk, the conformal mapping will lead to a weighted Steklov eigenvalue problem with an oscillatory weight function $\rho$. For some $\mathbf{x}$, $\rho(\mathbf{x})$ could approach to zero or infinity. With our box constraints, our optimal eigenvalues $\lambda_k \gamma$ will be bounded above by $2k\pi$. In Figure \ref{fig:Disk max 1}, we show two nonconstant maximizers $\rho$ for $\lambda_1 \gamma$ obtained numerically using our algorithm with different $\alpha$ and $\beta$. Different optimal $\rho$ can be obtained with different initial guesses for $\rho$. In our simulations, we choose $1+0.2 \cos (\theta)$.  

\begin{figure}
	\centering
	\includegraphics[scale=0.75]{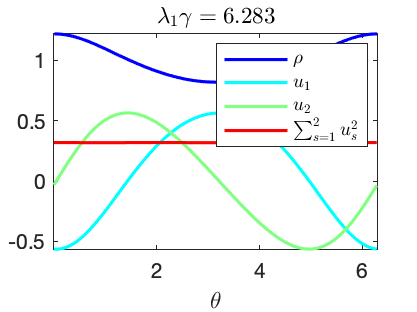}\includegraphics[scale=0.75]{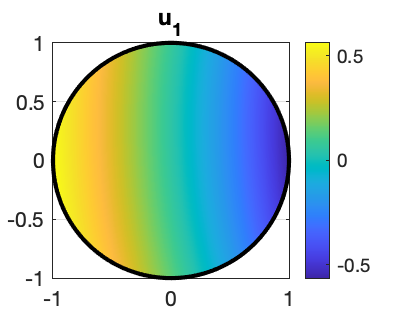}\includegraphics[scale=0.75]{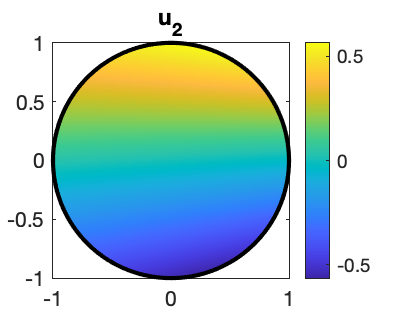}\\
    \includegraphics[scale=0.75]{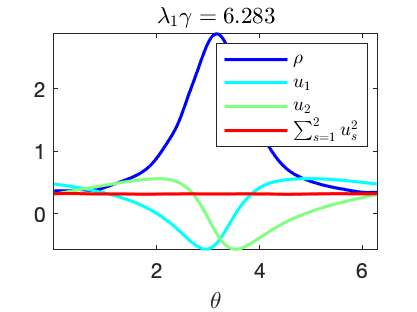}\includegraphics[scale=0.75]{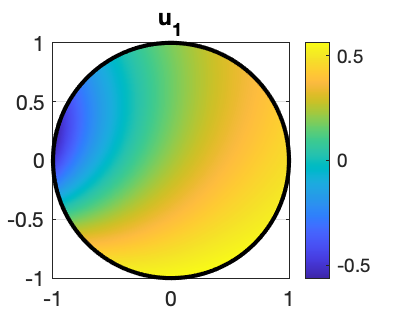}\includegraphics[scale=0.75]{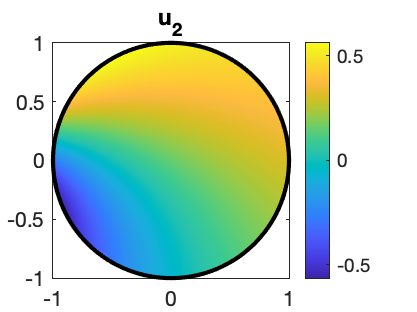}\\
  	\caption{\label{fig:Disk max 1} (a) The nonconstant maximizer of $\lambda_{1} \gamma$ on a disk for $\alpha = 0.5, \beta = 1.5$, and $\gamma = 2\pi.$ (b) The nonconstant maximizer of $\lambda_{1} \gamma$ on a disk for $\alpha = 0.25, \beta = 4$, and $\gamma = 2\pi.$}	
\end{figure}

The numerical optimizers satisfy the optimality conditions derived in
Theorems~\ref{thm:OCmin} and~\ref{thm:OCmaxandexist}. In the minimization problem the optimal density assigns its larger value to regions where the associated  function $w(\mathbf{x})$ is larger, whereas in the
maximization problem the larger density is placed where $w(\mathbf{x})$ is smaller. This
behavior is clearly visible in Figures~\ref{fig:Disk min}-\ref{fig:Disk max 1}.

\begin{rem}[Consistency of numerical results with the analysis]
The numerical experiments confirm our analytical findings. In the minimization problem the optimizer is of bang--bang type, in agreement with Theorem~\ref{thm:existmin}. Moreover, for both minimization and maximization the computed optimizers satisfy the  optimality conditions obtained in Theorems \ref{thm:OCmin}-\ref{thm:OCmaxandexist} expected of an optimizer. In the minimization case we also observe that the set where the optimizer equals $\beta$ is disconnected, with at least $m$ connected components corresponding to the number of zeros of the associated eigenfunctions, see Corollary \ref{cor:disconnectedOptimSet}. For the maximization problem with $k=1$, the optimizer is the constant function, in accordance with Theorem~\ref{thm:MaxExistAnalSol}. Finally, in this case we also have several  maximizers with different profiles, which is consistent with Remark~\ref{rem:distmaximizers}.
\end{rem}
\begin{rem}[Limitations of the numerical method]
While the method of particular solutions provides an efficient tool for the disk and other simply connected planar domains, it also has limitations. The approach relies on explicit harmonic basis functions in polar coordinates, which restricts its applicability to domains where such expansions are convenient. For more general geometries, constructing appropriate basis functions or ensuring accurate sampling along $\partial \Omega$ becomes more involved, and the conditioning of the matrices $A$ and $B$ may deteriorate. In higher dimensions, the method could be less practical, since the analogue of the Fourier--harmonic basis may lead to large-scale, ill-conditioned systems. For such cases, finite element or boundary element methods may provide more robust alternatives, albeit at a higher computational cost.
\end{rem}
\begin{rem}
The minimization and maximization problems considered in this paper may admit
multiple global optimizers. Such non-uniqueness is intrinsic and is common in
optimization problems for eigenvalues and optimal shapes, see for example \cite{henrot2006extremum}. It may arise from
different mechanisms: on symmetric domains such as the disk, rotational invariance
implies that any rotation of an optimizer is again optimal, while more generally the
threshold structure induced by the optimality conditions can lead to  densities attaining the same extremal value. The numerical algorithm is
therefore not intended to select a distinguished optimizer, but converges to one
admissible solution satisfying the optimality conditions; different initializations
may lead to different, but equally optimal, density profiles.
\end{rem}

\section{Conclusion}\label{sec:conc}

In this work we investigated the extremal behaviour of Steklov eigenvalues with respect to boundary densities under pointwise and integral constraints. We established existence of both minimizers and maximizers and provided structural results: minimizers are bang--bang functions and may have disconnected supports, while maximizers need not be bang--bang. A particular novelty arises in the maximization problem on circular domains, where we showed that infinitely many maximizers exist but are not generated by symmetry; instead, they originate from the conformal automorphisms of the disk. This stands in contrast to the minimization problem, where non-uniqueness is purely symmetry--induced.  

From an analytical standpoint, we demonstrated that the eigenvalue maps $\rho \mapsto \lambda_k(\rho), 1/\lambda_k(\rho)$ are neither convex nor concave in general, clarifying limitations of convex optimization tools and highlighting the need for problem-specific methods. To this end, we introduced a surrogate Fr\'echet differentiable functional which allowed us to derive optimality conditions in both minimization and maximization settings. Building on this framework, we proposed efficient numerical algorithms and explored a variety of examples. The numerical results reveal both the promise and the difficulty of computing optimal densities, particularly when the optimizer lacks smoothness or exhibits a disconnected structure.  

Our study contributes to the broader field of spectral optimization by opening up new perspectives on Steklov-type problems beyond the classical focus on shape optimization. Several directions remain for future work. These include developing numerical methods to solve the optimization problems in higher dimensions and complex geometries, studying the topological and geometric properties of the optimizers, and designing more robust schemes for non-smooth densities. Another promising line of research lies in connecting the optimization of Steklov eigenvalues with related inverse problems and applications in biomedical imaging and physics, where boundary effects play a dominant role.

\section*{Appendix A: Eigenvalues of Dirichlet Laplacian} \label{App:eigenvalue}
Consider the \( k \)th eigenvalue of the weighted Dirichlet problem:
\begin{equation}\label{eq:pdedirichlet}
	-\Delta u = \mu\, \rho\, u, \quad \text{in } \Omega, \quad u = 0 \quad \text{on } \partial \Omega,
\end{equation}
where \( \rho \) is a non-negative bounded function and \( \mu_k(\rho) \) denotes the \( k \)th eigenvalue. In \cite[ Theorem 9.1.3]{henrot2006extremum}, it is claimed that the map \( \rho \mapsto 1/\mu_k(\rho) \) is convex for all $k \geq 1$. However, upon closer inspection, the reasoning applies only to the principal eigenvalue \( 1/\mu_1(\rho) \), and does not extend to higher eigenvalues. 

In what follows, we present a one-dimensional example that demonstrates the failure of convexity (and concavity) for the Dirichlet eigenvalues \( \mu_k(\rho) \) and also $1/\mu_k(\rho)$.

 Consider 
\[
-u^{\prime\prime}(x)=\mu\rho(x) u(x),\quad x\in(0,1),\quad u(0)=u(1)=0.
\]
Denote
\[
\rho_{1}(x)=\begin{cases}
	\beta & x<\frac{1}{2},\\
	\alpha & x\ge\frac{1}{2},
\end{cases}\quad\text{and}\quad\rho_{2}=\begin{cases}
	\alpha & x<\frac{1}{2},\\
	\beta & x\ge\frac{1}{2},
\end{cases}
\]
and $\rho_{t}=t\rho_{1}+(1-t)\rho_{2}$ for $0\le t\le1.$ We have
\[
\rho_{t}(x)=\begin{cases}
	\rho_{L}:=\alpha+t(\beta-\alpha) & x<\frac{1}{2},\\
	\rho_{R}:=\beta-t(\beta-\alpha) & x\ge\frac{1}{2},
\end{cases}
\]
The solution $u$ with $\rho=\rho_{t}$ is given by 
\[
u(x)=\begin{cases}
	c_{1}\sin(\sqrt{\mu\rho_{L}}x) & x<\frac{1}{2},\\
	c_{2}\sin(\sqrt{\mu\rho_{R}}(1-x)) & x\ge\frac{1}{2}.
\end{cases}
\]
The eigenvalue is determined by the continuity conditions of $u$
and $u^{\prime}$ at $\frac{1}{2}$ which leads to 
\begin{align*}
&\det\left[\begin{array}{cc}
	\sin\left(\frac{1}{2}\sqrt{\mu\rho_{L}}\right) & -\sin\left(\frac{1}{2}\sqrt{\mu\rho_{R}}\right)\\
	\sqrt{\mu\rho_{L}}\cos\left(\frac{1}{2}\sqrt{\mu\rho_{L}}\right) & \sqrt{\mu\rho_{R}}\cos\left(\frac{1}{2}\sqrt{\mu\rho_{R}}\right)
\end{array}\right] \\
&= \sqrt{\mu\rho_{R}}\sin\left(\frac{1}{2}\sqrt{\mu\rho_{L}}\right)\cos\left(\frac{1}{2}\sqrt{\mu\rho_{R}}\right)+\sqrt{\mu\rho_{L}}\sin\left(\frac{1}{2}\sqrt{\mu\rho_{L}}\right)\cos\left(\frac{1}{2}\sqrt{\mu\rho_{R}}\right)=0.
\end{align*}
Roots  of this equation provide the eigenvalues corresponding to $\rho_t$. Setting $\alpha = 1$ and $\beta =2$,
Figure \ref{fig:The-plots-of first four eigenvalues} shows the plot of the first four eigenvalues. We can clearly see that only $\mu_{1}(\rho_{t})$ shows concavity, while $\mu_3(\rho_{t})$ and $\mu_4(\rho_{t})$ are neither convex nor concave. Correspondingly, $1/ \mu_{1}(\rho_{t})$ shows convexity, while $1/\mu_3(\rho_{t})$ and $1/\mu_4(\rho_{t})$ are neither convex nor concave. Even though $\mu_2$ seems convex along the path $\rho_t$, we can construct an example that this does not hold true.   
Denote
\[
\rho_{3}(x)=\begin{cases}
	\beta & |x-\frac{1}{2}|<\frac{1}{4},\\
	\alpha & |x-\frac{1}{2}|\ge\frac{1}{4},
\end{cases}\quad\text{and}\quad\rho_{4}(x)=\begin{cases}
	\alpha & |x-\frac{1}{2}|<\frac{1}{4},\\
	\beta & |x-\frac{1}{2}|\ge\frac{1}{4},
\end{cases}
\]
and define $\varrho(t) = t\rho_{3}+(1-t)\rho_{4}$ for $0\le t\le1.$ The eigenvalues $\mu_1$ and $\mu_2$ are shown in Figure \ref{f:D1_lam222}. Along the path $\varrho_t$, $\mu_1$ is no longer concave and $\mu_2$ is no longer convex. Thus we conclude $\mu_1$ and $\mu_2$ are neither convex nor concave. However, $1/\mu_1$ remains convex.

\begin{figure}
	\centering
	\includegraphics[scale=0.7]{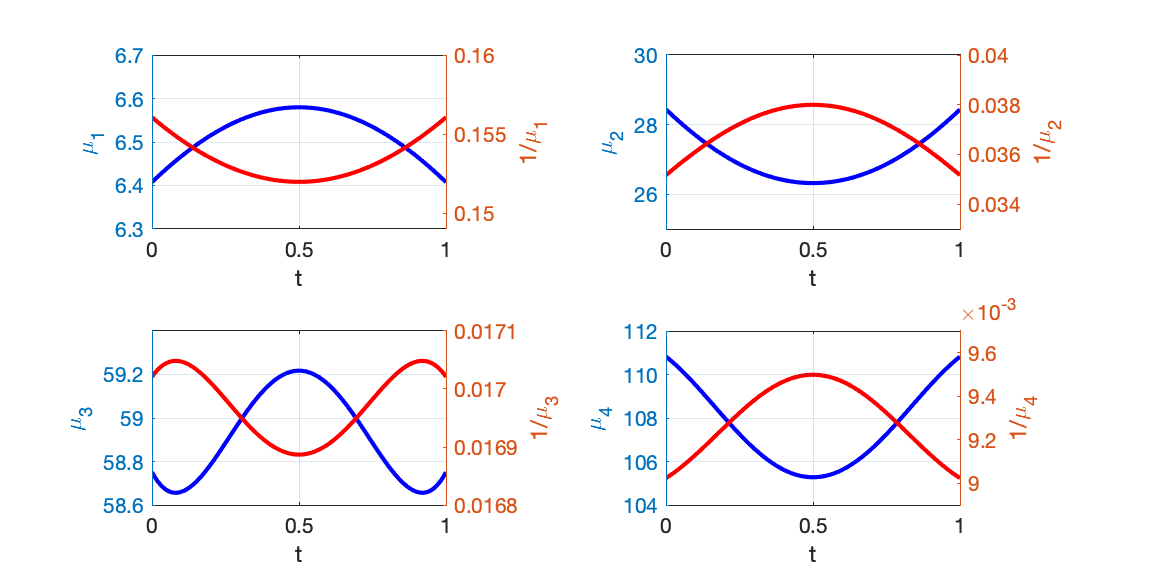}		\caption{\label{fig:The-plots-of first four eigenvalues}The plots of $\mu_{k}(\rho_t)$ and $1/\mu_{k}(\rho_t)$ for $k=1\cdots4$, respectively.
		 While $\mu_{1}(\rho_t)$ is concave or $1/\mu_{1}(\rho_t)$ is convex,  $\mu_2(\rho_t)$ is convex, $\mu_3(\rho_t)$ and $\mu_4(\rho_t)$ are neither convex nor concave.}	
\end{figure}

\begin{figure}
	\centering
	\includegraphics[scale=0.65]{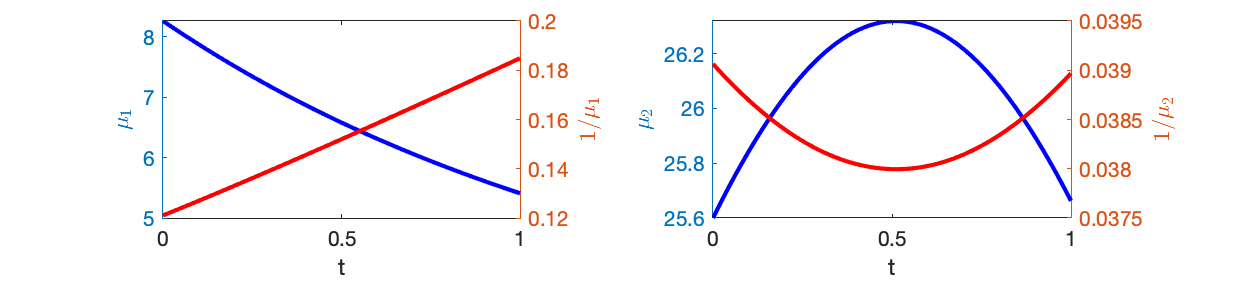}		\caption{\label{f:D1_lam222}The plots of $\mu_{k}(\varrho_t)$ and $1/\mu_{k}(\varrho_t)$ for $k=1$ and $k=2$, respectively. Together with Figure \ref{fig:The-plots-of first four eigenvalues}, we conclude that $\mu_1$ and $\mu_2$ are neither convex nor concave.}
\end{figure}

\section*{Appendix B: Proof of Lemma~\ref{lem:bathtublem} (minimization case)} \label{App:bathtub}

\begin{proof}
We present the proof only for the minimization problem; the maximization case follows by the same argument  and is omitted for brevity.

\medskip
\noindent
\textit{Step 1: Reduction to an equivalent problem.}
For $\rho\in\mathcal{M}$, set
\[
\sigma:=\frac{\rho-\alpha}{\beta-\alpha} \in [0,1] \quad \text{a.e. on } \partial\Omega,
\]
so that $\int_{\partial\Omega} \sigma\,dS = A$. Conversely, any $\sigma$ with these properties yields an admissible $\rho$.
The cost functional becomes
\[
\int_{\partial\Omega} \rho f\,dS 
= \alpha\int_{\partial\Omega} f\,dS + (\beta-\alpha)\int_{\partial\Omega} \sigma f\,dS.
\]
Since the first term is constant in $\sigma$, the problem reduces to
\begin{equation}\label{eq:min-sigma}
\min \left\{ \int_{\partial\Omega} \sigma f\,dS : \ 0\le \sigma\le 1 \ \text{a.e.}, \ \int_{\partial\Omega} \sigma\,dS = A \right\}.
\end{equation}

\medskip
\noindent
\textit{(i) Existence.}
The feasible set in \eqref{eq:min-sigma} is convex, closed, and bounded in $L^\infty(\partial\Omega)$, hence weak-* compact by the Banach–Alaoglu theorem. 
The map $\sigma \mapsto \int \sigma f$ is weak-* continuous because $f \in L^1(\partial\Omega)$.
Thus a minimizer $\sigma^\star$ exists, and so does $\rho^\star = \alpha + (\beta-\alpha)\sigma^\star \in \mathcal{M}$.

\medskip
\noindent
\textit{(ii) Pointwise characterization.}
We claim that $\bar\rho = \alpha+(\beta-\alpha)\bar\sigma$ is a minimizer if and only if
\[
\bar\sigma(\mathbf{x}) =
\begin{cases}
1, & f(\mathbf{x})<\tau,\\
\in[0,1], & f(\mathbf{x})=\tau,\\
0, & f(\mathbf{x})>\tau,
\end{cases}
\quad\text{i.e.}\quad
\bar\rho(\mathbf{x}) =
\begin{cases}
\beta, & f(\mathbf{x})<\tau,\\
\in[\alpha,\beta], & f(\mathbf{x})=\tau,\\
\alpha, & f(\mathbf{x})>\tau.
\end{cases}
\]

\emph{Necessity.}
Let $\bar\rho\in\mathcal M$ be a minimizer. Set
\[
L:=\{x\in\partial\Omega:\ f(\mathbf{x})<\tau\},\qquad
E:=\{x\in\partial\Omega:\ f(\mathbf{x})=\tau\},\qquad
G:=\{x\in\partial\Omega:\ f(\mathbf{x})>\tau\}.
\]

\smallskip
We show that $\bar\rho=\alpha$ a.e.\ on $G$.
Assume, towards a contradiction, that there exists a measurable $H\subset G$ with $|H|>0$ and $\bar\rho>\alpha$ on $H$. 
We show there is a measurable $F\subset L\cup E$ with $|F|>0$ and $\bar\rho<\beta$ on $F$. 
If not, then $\bar\rho=\beta$ a.e.\ on $L\cup E$, hence
\[
\int_{\partial\Omega}\bar\rho\,dS
=\int_{G}\bar\rho\,dS+\int_{L\cup E}\bar\rho\,dS
>\alpha\,|G|+\beta\,|L\cup E|
=\alpha|\partial\Omega|+(\beta-\alpha)|L\cup E|
\ge \gamma,
\]
since $|L\cup E|\ge A=(\gamma-\alpha|\partial\Omega|)/(\beta-\alpha)$; this contradicts $\int\bar\rho=\gamma$. 
Thus such an $F$ exists.

Define $\tilde\rho$ by
\[
\tilde\rho(\mathbf{x})=
\begin{cases}
\alpha, & \mathbf{x}\in H,\\
\beta, & \mathbf{x}\in F,\\
\bar\rho(\mathbf{x}), & \text{otherwise},
\end{cases}
\]
with $F$ chosen so that the mass is preserved:
\begin{equation}\label{eq:mass-balance}
\int_{\partial\Omega}(\bar\rho-\tilde\rho)\,dS
=\int_H(\bar\rho-\alpha)\,dS+\int_F(\bar\rho-\beta)\,dS
=0,
\end{equation}
which is possible because $\bar\rho-\alpha>0$ on $H$ and $\bar\rho-\beta<0$ on $F$.
Then $\tilde\rho\in\mathcal M$.

Compute the cost difference:
\[
\int_{\partial\Omega}(\bar\rho-\tilde\rho)\,f\,dS
=\int_H(\bar\rho-\alpha)f\,dS+\int_F(\bar\rho-\beta)f\,dS.
\]
On $H\subset G$, $f>\tau$, hence $(\bar\rho-\alpha)f>(\bar\rho-\alpha)\tau$. 
On $F\subset L\cup E$, $f\le\tau$ and $(\bar\rho-\beta)\le 0$, thus $(\bar\rho-\beta)f\ge(\bar\rho-\beta)\tau$. Therefore,
\[
\int_{\partial\Omega}(\bar\rho-\tilde\rho)\,f\,dS
>\tau\!\int_H(\bar\rho-\alpha)\,dS+\tau\!\int_F(\bar\rho-\beta)\,dS
=\tau\!\int_{\partial\Omega}(\bar\rho-\tilde\rho)\,dS
=0,
\]
by \eqref{eq:mass-balance}. Hence
\[
\int_{\partial\Omega}\tilde\rho\,f\,dS
<\int_{\partial\Omega}\bar\rho\,f\,dS,
\]
contradicting the optimality of $\bar\rho$. Thus $\bar\rho=\alpha$ a.e.\ on $G$.

\smallskip
Similarly, one can show that $\bar\rho=\beta$ a.e.\ on $L$; we omit the proof for brevity.
Therefore,
this yields the pointwise rule in Lemma~\ref{lem:bathtublem}(ii).

\emph{Sufficiency.}
Let $\bar\sigma$ satisfy the pointwise rule from (ii) and the mass constraint $\int_{\partial\Omega}\bar\sigma\,dS=A$.
For any feasible $\sigma$, we have
\[
\int_{\partial\Omega}(\sigma-\bar\sigma)f\,dS
\;\ge\;
\tau\!\left[\int_{L}(\sigma-1)\,dS + \int_{E}(\sigma-\bar\sigma)\,dS + \int_{G}\sigma\,dS\right]
=\tau\!\int_{\partial\Omega}(\sigma-\bar\sigma)\,dS
=\tau\,(A-A)=0,
\]
which implies $\int_{\partial\Omega}\sigma f\,dS \ge \int_{\partial\Omega}\bar\sigma f\,dS$.
Hence $\bar\sigma$ minimizes the cost.

\medskip
\noindent
\textit{(iii) Existence of a bang–bang minimizer.}
By the definition of $\tau$,
\[
|\{f<\tau\}|\ \le\ A\ \le\ |\{f\le\tau\}|.
\]
Thus there exists a measurable set $E$ such that
\[
\{f<\tau\} \subset E \subset \{f\le\tau\}, \quad |E|=A.
\]
Let $\bar\sigma = \chi_E$ and $\bar\rho = \alpha + (\beta-\alpha)\chi_E$. Then $\bar\sigma$ satisfies the pointwise rule in (ii), hence $\bar\rho$ is a minimizer of bang–bang form.

\medskip
\noindent
\textit{(iv) Uniqueness implies full level set.}
If the minimizer is unique, then necessarily $|\{f=\tau\}|=0$ or the mass constraint forces $E$ to contain all of $\{f=\tau\}$. Otherwise, different choices of $E$ on $\{f=\tau\}$ would yield distinct minimizers of the same cost. Therefore
\[
E = \{ f \le \tau \} \quad\text{(up to a null set)},
\]
and the unique minimizer is $\bar\rho = \alpha + (\beta-\alpha)\chi_{\{f\le\tau\}}$.

\smallskip
This completes the proof of Lemma~\ref{lem:bathtublem} for the minimization case.
\end{proof}

\section*{Statements \& Declarations}
\textbf{Funding}: Chiu Yen Kao acknowledges support from NSF DMS-2208373 and DMS-2513176.\\
\textbf{Conflict of interests}: The authors declare that there are no conflicts of interest regarding the publication of this paper.\\
\textbf{Data availability}:  Data will be made available on request.\\



\bibliographystyle{plain}
\bibliography{Steklov_Kao_Mohammadi_ref}

\end{document}